\documentclass[11pt]{article}
\usepackage{fullpage}

\usepackage{amsmath, amssymb, epsfig, color, graphicx}

\newcommand{\relgap}{{\rm rel\_gap}}

\def\allow{\mathop{\rm Allow}\nolimits}
\def\qallow{\mathop{\rm q\!-\!Allow}\nolimits}
\def\Int{\mathop{\rm Int}\nolimits}

\newcommand{\actaqed}{\hfill $\Box$}

\newenvironment{proof}{Proof: }{$\Box$}

{\medskip\noindent \textit{Sketch of Proof for #1.}}%
{\actaqed \medskip}

{\medskip\noindent \textit{Proof of #1.}}%
{\actaqed \medskip}

\newcommand{\Rr}{\mathbb{R}}
\newcommand{\C}{\mathbb{C}}
\newcommand{\G}{\mathcal{G}}
\def\S{\mathcal{S}}
\def\k{{\cal K}}
\def\ii{{\cal I}}

\newcommand{\eqbd}{\mathop{:}{=}}

\newcommand{\bmat}{\left[ \begin{array}}
\newcommand{\emat}{\end{array} \right]}

\newtheorem{theorem}{Theorem}[section]

\newtheorem{lemma}[theorem]{Lemma}
\newtheorem{corollary}[theorem]{Corollary}

\newtheorem{definition}[theorem]{Definition}
\newtheorem{example}[theorem]{Example}
\newcommand{\ignore}[1]{}
\newcommand{\hugeX}{\phantom{\huge X}\!\!\!\!\!\!\!\!}
\newcommand{\diag}{{\rm diag}}

\begin{document}

\title{Accurate and Efficient Expression Evaluation and Linear Algebra}
\author{James Demmel\thanks{Department of Mathematics and Computer Science Division, 
University of California-Berkeley.  J. Demmel acknowledges support of  NSF under grants CCF-0444486,
 CNS 0325873, by DOE grant DE-FC02-06ER25786, and of the University of California-Berkeley  
Richard Carl Dehmel Distinguished Professorship.}, Ioana Dumitriu\thanks{Department of Mathematics,
University of Washington.  I. Dumitriu acknowledges support of the Miller Institute
for Basic Research in Science.}, Olga Holtz\thanks{Departments of Mathematics, University of California-Berkeley
and Technische Universit\"at Berlin. O. Holtz acknowledges support of the Sofja
Kovalevskaja programm of Alexander von Humboldt Foundation.}  and 
Plamen Koev\thanks{Department of Mathematics, North Carolina State University. 
P. Koev acknowledges support of NSF under grants DMS-0314286, DMS-0411962, 
DMS-0608306.}}

\maketitle
\label{firstpage}

\begin{abstract}
We survey and unify recent results on the existence of accurate algorithms
for evaluating multivariate polynomials, and more generally for accurate numerical
linear algebra with structured matrices. By "accurate" we mean that the computed
answer has relative error less than $1$, i.e., has some correct leading digits.
We also address efficiency, by which we mean algorithms that run in polynomial
time in the size of the input. Our results will depend strongly on the model
of arithmetic: Most of our results will use the so-called {\em Traditional Model (TM),\/}
where the computed result of $op(a,b)$, a binary operation like $a+b$, is given by
$op(a,b)*(1 + \delta)$ where all we know is that $| \delta | \leq \varepsilon \ll 1$. 
Here $\varepsilon$ is a constant also known as machine epsilon.

We will see a common reason that the following disparate problems all permit
accurate and efficient algorithms using only the four basic arithmetic operations:
finding the eigenvalues of a suitably discretized scalar elliptic PDE,
finding eigenvalues of arbitrary products, inverses, or Schur complements
of totally nonnegative matrices (such as Cauchy and Vandermonde), and evaluating
the Motzkin polynomial. Furthermore, in all these cases the high accuracy is "deserved", 
i.e., the answer is determined much more accurately by the data than the 
conventional condition number would suggest. 

In contrast, we will see that evaluating even the simple polynomial
$x+y+z$ accurately is impossible in the TM, using only the 
basic arithmetic operations. We give a set of necessary and sufficient conditions to decide
whether a high accuracy algorithm exists in the TM, and describe progress
toward a decision procedure that will take any problem and provide either  a high
accuracy algorithm or a proof that none exists.

When no accurate algorithm exists in the TM, it is natural to extend the set of 
available accurate operations by a library of additional operations, such
as $x+y+z$, dot products, or indeed any enumerable set which could then be
used to build further accurate algorithms. We show how our accurate algorithms 
and decision procedure for finding them extend to this case.

Finally, we address other models of arithmetic, and the relationship
between (im)possibility in the TM and (in)efficient algorithms operating 
on numbers represented as bit strings.
\end{abstract}

\newpage

\tableofcontents

\section{Introduction}  \label{sec_intro}

A result of a computation will be called {\em accurate\/} 
if it has a small relative error, in particular less than $1$ 
(i.e., some leading digits must be correct).
Now we can ask what the following problems have in common:
\begin{enumerate}
\item Accurately evaluate the Motzkin polynomial 
$$p(x,y,z) = z^3 + x^2 y^2 ( x^2 + y^2 - 3z^2)~.$$
\item Accurately compute the entries or eigenvalues of 
a matrix gotten by performing an
arbitrary sequence of operations chosen from the set
\{multiplication, $J$-inversion, Schur complement, taking submatrices\},
starting from a set of Totally Nonnegative (TN) matrices such as
the Hilbert matrix, TN generalized Vandermonde matrices, {\em etc.\/}
\item Accurately find the eigenvalues 
of a suitably discretized scalar elliptic PDE.
\end{enumerate}
We also ask how they all differ from the apparently much
easier problem of evaluating $x+y+z$.

The answer will depend strongly on our model of arithmetic.
For most of this paper we will use the 
{\em Traditional Model (TM)\/} of arithmetic, that
the computed result of $op(a,b)$, a binary operation
like $a+b$, is given by $op(a,b)\cdot (1 + \delta)$
where all we know is that $|\delta| \leq \varepsilon \ll 1$.
Here $\varepsilon$ is a real constant also known as {\em machine precision.\/}
We will refer to $rnd(op(a,b)) \equiv op(a,b)(1 + \delta)$ as the 
{\em rounded result\/} of $op(a,b)$.
We will distinguish between the cases where the other quantities (including $\delta$s) 
are all real, or all complex.

To see why some expressions may or may not be evaluable accurately in the TM,
consider multiplying or dividing two numbers each known to relative error $\eta < 1$: 
then their rounded product or quotient is clearly correct with relative
error $O(\max(\eta,\varepsilon))$. 
This also holds when adding two like-signed real numbers 
(or subtracting real numbers with opposite signs).
In contrast, subtracting two like-signed
real numbers $x-y$ can lead to {\em cancellation\/} of leading digits:
If $x$ and $y$ themselves have nonzero relative error bounds,
then depending on the extent of cancellation, $x-y$ may have
an arbitrary relative error.
On the other hand if $x$ and $y$ are exact inputs, then
$rnd(x \pm y)=(x \pm y)(1 + \delta)$ is also known with small relative error.
In other words, an easy sufficient (but not necessary!) 
condition in the TM for an algorithm to be accurate is 
``No Inaccurate Cancellation'':
\begin{description}
\item[NIC:] The algorithm only (1) multiplies, (2) divides, (3) adds (resp., subtracts) 
real numbers with like (resp., differing) signs, 
and otherwise only (4) adds or subtracts input data. 
\end{description}
Sometimes we will also include the square root among our
allowed operations in 
NIC\footnote{However, square roots require more care in bounding
the relative error: In floating point arithmetic on most computers, 
computing $y = x^{1/2^{100}}$
by $100$ square roots and then $z = y^{2^{100}}$ by $100$ squarings,
yields $z=1$ independently of $x > 0$.}.

In the TM, with real numbers, the three problem listed above all have
novel accurate algorithms that use only four basic arithmetic operations 
($+$, $-$, $\times$ and $/$), comparison and branching, and satisfy NIC.
Furthermore, the matrix algorithms are efficient, running in $O(n^3)$ time 
(we say more about efficiency below).
These linear algebra algorithms depend on some recently discovered matrix 
factorizations and update formulas, and the algorithm for the Motzkin 
polynomial (surprisingly) fills a page with $8$ cases. 
In contrast, with complex arithmetic, no accurate algorithms exist. 
Nor is there an accurate algorithm using only these operations, in the real
or complex case, for evaluating $x+y+z$ accurately.

For example, consider Figure~\ref{figur1}, which shows the eigenvalues 
of a matrix gotten by taking the trailing $20$-by-$20$ Schur complement of
a $40$-by-$40$ Vandermonde matrix. Both the eigenvalues computed
by our algorithm (in standard double precision floating point
arithmetic), and by a conventional algorithm are 
shown. Note that {\em every\/} eigenvalue computed by the
conventional algorithm is wrong by orders of magnitude,
whereas all ours are correct to nearly $14$ digits, as confirmed 
by a very high precision calculation. 

\begin{figure}
\begin{center}
{\epsfxsize=8.75cm \epsfbox{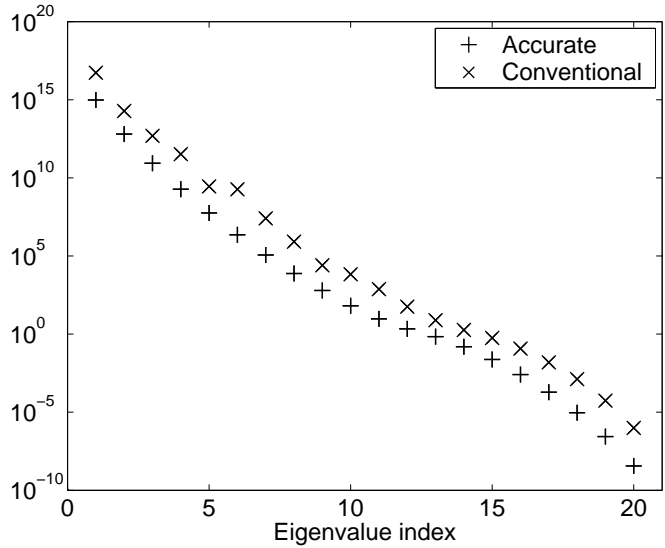}}  
\caption{Eigenvalues of the $20$th Schur Complement of the 
$40$-by-$40$ Vandermonde matrix $V_{ij} = i^{j-1}$, computed both 
using a Conventional algorithm (x) and and Accurate algorithm (+).}
\label{figur1}
\end{center} 
\end{figure}

Section~\ref{sec_Plamen} of this paper will survey a great many other examples of structured 
matrices where accurate and efficient linear algebra algorithms are possible 
using NIC as the main (but not only) tool; see Table \ref{table1} for a summary. 

One may wonder whether this accuracy is an ``overkill'', because small uncertainties 
in the data might cause much larger uncertainties in the computed results.
In this case computing results to high accuracy would be more than the data 
deserves, and not worth any additional cost.
Indeed the usual condition numbers of the problems
considered here are usually enormous. However, their
{\em structured\/} condition numbers are often quite
modest, justifying computing the answers to high accuracy. 
For example, while a Cauchy matrix 
$C_{ij} = 1/(x_i + y_j)$ such as the Hilbert
matrix ($x_i = i = 1 + y_i$) is considered
badly conditioned since $\kappa(C) \equiv \|C\| \cdot \|C^{-1}\|$
can be very large, the entries of $C^{-1}$ are actually much less 
sensitive functions of $x_i$ and $y_j$ than $\kappa (C)$ would indicate.
Indeed, if the answer is given by a formula satisfying NIC,
then the condition number can only be large when cancellation
occurs when computing $x \pm y$ for uncertain input data $x$
and $y$; each such expression adds the quantity
$1/\relgap(x,y) \equiv (|x| + |y|)/|x \pm y|$ 
to the structured condition number.  This is true of all the examples in 
Section~\ref{sec_Plamen},  justifying their more accurate computation than
would the usual condition number.

The profusion and diversity  of these examples naturally raises 
the question as to what mathematical property they share that
makes these algorithms possible. Section~\ref{sec_OlgaIoana} of this paper
addresses this, by describing progress towards a {\em decision procedure\/}
for the more basic problem of deciding whether a given multivariate polynomial 
can be evaluated  accurately using the basic rounded arithmetic operations,
comparison, and branching. 
The answer will depend not just
on the polynomial, but whether the data is real or complex,
and on the domain of evaluation (a smaller domain may be 
easier than a larger one, if it eliminates difficult arguments).
This decision procedure would yield
simpler necessary and sufficient conditions (not identical in
all cases) that tell us whether the algorithms in Section~\ref{sec_Plamen} 
(or others not yet discovered) must exist (we will use the
fact that accurate determinants are necessary and often sufficient
for accurate linear algebra). It will turn
out that the results for real arithmetic are much more
complicated than for complex arithmetic, where simple
necessary and sufficient condition may be stated 
(the answer is basically given by NIC above);
this reflects the difference between algebraic geometry over
the real and complex numbers.


One negative result of Section~\ref{sec_traditional} will be the impossibility
of evaluating $x+y+z$ using only the basic rounded arithmetic operations.
This seems odd, since $x+y+z$ is so simple. But it is only simple
if we use the fact that in practice (floating point arithmetic), 
$x$, $y$ and $z$ are represented by finite bit strings that can
be manipulated and analyzed differently than by assuming only that 
$rnd(op(a,b)) = op(a,b)(1 + \delta)$ with $|\delta| \leq \varepsilon$.
To go further we must extend our model of arithmetic.
We do so in two ways.

Section~\ref{sec_extended} continues by adding so-called ``black-box'' operations
to the basic arithmetic operations. For example, one could assume
that a subroutine for the accurate evaluation of $x+y+z$
(or of dot products, or of 3-by-3 determinants, {\em etc.\/})
also existed, and then ask the analogous question as to what
other polynomials could be accurately evaluated, using this
subroutine as a building block.  This indeed
models computational practice, where subroutine libraries
of such black-box routines are provided in order to build
accurate algorithms for other more complicated polynomials. 
In Section~\ref{sec_extended} we also describe how to extend our decision procedures 
when an arbitrary set of such black-box routines is available,
and the question is whether another polynomial not
already in the set can be evaluated accurately.
A positive result will be showing that just the ability to
compute $2$-by-$2$ determinants accurately is enough to permit accurate
and efficient linear algebra on the inverses of tridiagonal matrices. 
A negative result will be the impossibility
of accurate linear algebra with Toeplitz
matrices, given {\em any\/} set of block-box operations of
bounded degree or with a bounded number of arguments.

Sections~\ref{sec_traditional} and \ref{sec_extended} go some way to describing the possibilities
and limits of solving numerical problems accurately in
practice. But ``in practice'' means using finite representations
with bits, i.e., floating point, in which case accurate (even exact) polynomial
evaluation is always possible, and the only question is cost.
In Section~\ref{sec_OtherModels}, after a brief discussion of other arithmetic 
models, we will settle on one model we believe best captures the 
spirit of actual floating point computation, but without limiting 
it to fixed word sizes: an arbitrary pair of integers
$(m,e)$ is used to represent the floating point number $m \cdot 2^e$.
In this model, we describe how the algorithms in Section~\ref{sec_Plamen}
lead to efficient algorithms that run in time polynomial in
the size of the inputs, the usual computer science
notion of efficiency. In contrast, conventional algorithms simply
run in high enough precision to get an accurate answer do not run
in polynomial time.

Finally, in Section~\ref{sec_Conditioning} we consider the structured condition
numbers for the problems we consider, which can be much
smaller that the usual unstructured condition numbers and
so justify accuracy computation.
In prior work \cite{demmel87}, the first author observed 
that for many problems the condition number 
of the condition number was approximately equal to the
condition number of the original problem, 
and that this corresponded to the
geometric property that the condition number 
was the reciprocal of the distance to
the nearest ill-posed (or singular) problem. 
These observations apply here, with the
following interesting consequence: 
for the examples considered here it is
possible to compute the solution to a 
problem accurately if and only if it is possible
to estimate its condition number accurately.
An analogous phenomenon was already observed  
in \cite{demmeldiamentmalajovich}.

\section{Accurate and efficient algorithms for linear algebra}
\label{sec_Plamen}

\ignore{

\subsection{Outline (based on SIAM Review outline from Aug 25,
2005(!); statements of main theorems, at
most sketches of proofs, a few graphs of numerical examples)}

\begin{itemize}
\item  Goals (columns of table: inv(A), LDU, SVD, etc)
\item Basic tools
\begin{itemize}
\item  No (or bounded) cancellation of inexact quantities
\item  Bidiagonal SVD
\item  One-sided Jacobi for SVD
\item   RRD, its SVD
\item   Graded matrices (anticipate DSTU)
\item   SVD for symmetric eigenproblem (def and indef?)
\item   Necessity, sufficiency of determinants
\end{itemize}
\item   Filling in rows of table
\begin{itemize}
\item   If only sparsity pattern known (acyclic)
\item   If sparsity and sign pattern known (TSC)
\item   Diagonally dominant (M-matrix or not)
\item   Unit displacement rank (Cauchy, (poly) Vandermonde)
\item   Differential operators (DSTU \& applications)
\item   Totally positive
\end{itemize}
\end{itemize}

}

\subsection{Introduction} \label{sec_P_intro}

The numerical linear algebra problems we
will consider include computing the product of matrices,
the Schur complement, the determinant or other minor,
the inverse, the solution to a linear system or least
squares problem, and various matrix decompositions
such as LDU (with or without pivoting)
QR, SVD (singular value decomposition), and
EVD (eigenvalue decomposition).

Conventional algorithms for these problems are at best only
{\em backward stable:\/} When applied to a matrix $A$
they compute the exact solution of a nearby problem
$A + \delta A$, where $\| \delta A \| = \mathcal O( \varepsilon ) \|A\|$,
where $\| \cdot \|$ is some matrix norm
and $\varepsilon$ is machine epsilon.
In consequence, the error in the computed solution
depends on how sensitive the answer is to small
changes in $A$, and is typically bounded in
norm by $\frac{\| \delta A \|}{\| A \|} \kappa (A)
= \mathcal O( \varepsilon ) \kappa (A)$, where
$\kappa (A)$ is a condition number (a scaled
norm of the Jacobian of the solution map).
Thus we have
two ways to lose high relative accuracy:
First, bounding the error only in norm may provide
very weak bounds for tiny solution components;
for example the error bound for the computed
singular values guarantees an absolute error
$| \sigma_{i,\mbox{true}} - \sigma_{i,\mbox{comp}} | = \mathcal O( 
\varepsilon ) \max_i 
\sigma_{i,\mbox{true}}$,
so that the large singular values have small relative
errors, but not the small ones.
Second, when $\kappa (A)$ is large, even large solution components
may be inaccurate, as when inverting an ill-conditioned matrix.

However, these conventional algorithms ignore the
{\em structure\/} of the matrix, which is critical to
our approach. Rather than treating, say, a Cauchy
matrix $C$ as a collection of $n^2$ independent
entries $C_{ij} = 1/(x_i + y_j)$, we treat it as
a function of its $2n$ parameters $x_i$ and $y_j$.
Starting from these $2n$ parameters, we can find
accurate expressions (because they satisfy NIC)
for $C$'s determinant
$\det (C) = \prod_{i<j} (x_i - x_j)(y_i - y_j)/ \prod_{i,j} (x_i + y_j)$
and other linear algebra problems.
As mentioned in Section~\ref{sec_intro}, expressions satisfying NIC
also imply that their structured condition numbers can be arbitrarily
smaller than their conventional condition numbers.

Now we outline our general approach to these problems.
First we consider the problems whose solutions
are rational functions of the parameters, such as
computing a determinant or minor.
Indeed, all these solutions can be
expressed using minors or quotients of minors. For example,
the entries of the inverse or LDU factorization are
(quotients of) minors,
the product $AB$ can be extracted from
$$
\bmat{ccc} I & A & 0 \\ 0 & I & B \\ 0 & 0 & 1 \emat^{-1},
$$
and the last column of
$$\bmat{ccc} I & A & -b \\ A^T & 0 & 0 \\ 0 & 0 & 1 \emat^{-1}
$$
contains the solution of the overdetermined least squares
problem $\min_x \| Ax-b \|_2$.
Thus the ability to compute certain minors
with high relative accuracy is sufficient to solve these
linear algebra problems with high relative accuracy.
Conversely, knowing a factorization like LDU
with high relative accuracy yields the determinant
with similar accuracy (via the product $\pm \prod_i D_{ii}$).
Thus we see that matrix structures that permit accurate
computations of certain determinants are both necessary
and sufficient
for solution of these linear algebra problems
with high relative accuracy.
In this section we will identify
a number of matrix structures that permit such accurate
determinants to be calculated.

Second, we consider the EVD and SVD,
which involve more general algebraic functions of the matrix entries.
To compute these accurately, we need other tools, which we
will summarize below in Section~\ref{sec_Tools}. Briefly, our approach will be to
compute one of several other matrix decompositions using only
rational operations (and possibly square roots), and
then apply iterative schemes to these decompositions that
have accuracy guarantees.

\ignore{
{\em Should we keep this paragraph or not?
I see Plamen has QR as a column in Table 1.}
The factorization $A=QR$ also requires square roots.
However the identity $A^TA = LDL^T = (LD^{1/2})(LD^{1/2})^T = R^TR$
implies that $R = D^{1/2} L^T$ and $Q = AR^{-1} = A L^{-T} D^{-1/2}$.
Thus we see that we only need to compute square roots of
(quotients of) minors (to get $D^{1/2}$) and
then multiply or divide rational functions
($L^T$ and $AL^{-T}$) by these square roots
to get the desired results ($R$ and $Q$).
}

Efficient conventional algorithms
(i.e., using $\mathcal 
O(n^3)$ arithmetic operations)
exist for each of the above problems and are
available in free packages
(e.g., LAPACK \cite{lapackmanual3}) or embedded in commercial
ones (e.g., MATLAB \cite{matlab}). So an extra challenge is to find
not just accurate algorithms, but ones that also take $\mathcal O(n^3)$
operations.

Our results, using only NIC, are summarized in Table~\ref{table1},
which describes (in a $\mathcal O( \cdot )$ sense) the speed of the 
fastest
known accurate algorithm for each problem shown.
There is one column for each linear algebra problem considered,
and one row for each structured matrix class. The abbreviations
not yet defined will be explained as we continue.

The rest of this section is organized as follows. Subsection 
\ref{sec_Tools} briefly presents accurate algorithms for the EVD and SVD. 
Subsection \ref{sec_design} walks through Table \ref{table1} row by row, 
again briefly explaining the results. Finally, Subsection 
\ref{sec_BeyondNIC} explains how much more is possible if we expand the 
class of formulas we may use beyond NIC in a certain disciplined way. This 
naturally raises the question of whether or not there is a systematic 
method to recognize such formulas, which is the final topic of 
this paper.

\subsection{Tools for computing EVD and SVD accurately}
\label{sec_Tools}

\subsubsection{Rank revealing decompositions and SVD}
\label{sec_RRD}
The first accurate SVD algorithm depends on
a {\em Rank Revealing Decomposition (RRD)\/} \cite{DGESVD} of matrix $A$,
a factorization $A = XDY$ where $D$ is nonsingular and diagonal,
and $X$ and $Y^T$ have full column rank and are ``well-conditioned''.
Note that $A$ may be rectangular or singular. The most obvious
example of an RRD is the SVD, where $X$ and $Y$ are as well-conditioned
as possible. Other examples where $X$ and $Y$ are (nearly always)
well-conditioned come from Gaussian elimination with complete pivoting
$A = LDU$, or from QR with complete pivoting $A = QDR$
(more sophisticated pivoting techniques with better condition
bounds on the unit triangular factors are available 
\cite{chan87,chandipsen,gueisenstat2,hongpan,hwly:92,ming,stewart93}).
An RRD $A=XDY$ has two attractive properties
\begin{enumerate}
\item
Given the RRD, it is possible to compute the SVD
to high relative accuracy in the following sense
\cite[Section 3]{DGESVD}, \cite[Algorithm 2]{demmelkoev99}:
\begin{itemize}
\item
The relative error in each singular value $\sigma_i$ is
bounded by $\mathcal O( \varepsilon \max ( \kappa (X) , \kappa (Y) ))$,
where $\kappa (X) = \| X \| \cdot \| X \|^{-1}$.
\item
The relative error in the $i$th computed (left or right)
singular vector is bounded by
$\mathcal O( \varepsilon \max( \kappa (X), \kappa (Y) ) / \min_{j \neq i} 
\relgap ( 
\sigma_i, \sigma_j )$.
In other words, the condition number can only be large if the singular value
agrees with another one to many leading digits, no matter how small they are
in absolute value.
\end{itemize}
\item
These error bounds do not change if the RRD
is known only approximately
(either because of uncertainty in $A$ or
roundoff in computing the RRD), as long as \cite[Theorem 2.1]{DGESVD},
\cite{eisenstatipsen,RCLi}:
\begin{itemize}
\item
We can compute $\hat{X}$ where $\| X - \hat{X} \| = \mathcal O( 
\varepsilon ) \| X \|$.
\item
We can compute a diagonal $\hat{D}$ where $|D_{ii} - \hat{D}_{ii}| = 
\mathcal O( \varepsilon ) |D_{ii}|$.
\item
We can compute $\hat{Y}$ where $\| Y - \hat{Y} \| = \mathcal O( 
\varepsilon ) \| Y \|$.
\end{itemize}
In other words, we only need the factors $X$ and $Y$ with high
absolute accuracy, not relative accuracy, a fact that will significantly
expand the scope of applicability.
\end{enumerate}

Among the various algorithms cited above for computing the
SVD, we sketch one \cite[Algorithm 3.2]{DGESVD},
along with an explanation of its accuracy:
\begin{tabbing}
junk \= junk \= junk \= junk \kill
\> 1) Compute the SVD of $XD$ using one-sided Jacobi,
yielding $XD = \bar{U} \bar{\Sigma} \bar{V}^T$.
Thus $A = \bar{U} \bar{\Sigma} \bar{V}^T Y $. \\
\> 2) Multiply $W = \bar{\Sigma} (  \bar{V}^T Y )$, respecting parentheses.
Thus $A = \bar{U}W$. \\
\> 3) Compute the SVD of $W$ using one-sided Jacobi, yielding
$W = \bar{\bar{U}} \Sigma V^T$. Thus $A = \bar{U} 
\bar{\bar{U}} \Sigma V^T$. \\
\> 4) Multiply $U = \bar{U} \bar{\bar{U}}$, yielding
 the SVD $A = U \Sigma V^T$.
\end{tabbing}

Briefly, the reason this works is that in steps 1) and 3), which
potentially combine numbers over very wide ranges of magnitude,
one-sided Jacobi respects this scaling by, in step 1) for example,
creating backward errors in column $i$ of $XD$ that are proportional
to $D_{ii}$ \cite{demmelveselic,drmac1995a,mathias96}.
Furthermore, each step costs $\mathcal O(n^3)$ arithmetic operations.

\subsubsection{Bidiagonal SVD}
\label{sec_BidiagonalSVD}
The second accurate SVD algorithm depends on a {\em Bidiagonal Reduction (BR)\/}
of matrix $A$,
a factorization $A = UBV^T$ where $B$ is bidiagonal
(nonzero on the main and first super-diagonal)
and $U$ and $V$ are unitary. This is an intermediate
factorization in the standard SVD algorithm.
If the entries of $B$ are determined
to high relative accuracy, so is $B$'s SVD in the
same sense as the RRD determines the SVD as described
above (but without any factor like $\max ( \kappa (X) , \kappa (Y) )$
in the error bounds).
Furthermore, accurate $\mathcal O(n^3)$ algorithms are available
\cite{demmelkahan,parlett95a}.

\subsubsection{Accurate EVD}
\label{sec_EVD}
Now we discuss the EVD. Clearly, if $A$ is symmetric positive definite, 
and a symmetric RRD $A = XDX^T$ is available, then the SVD and EVD are 
identical. If $A$ is symmetric indefinite but an accurate SVD is 
attainable, then the only remaining task is assigning correct signs 
to the singular values, which may be done using the algorithms of Dopico, 
Molera, and Moro \cite{dopicomoleramoro03}. Algorithms for computing 
symmetric RRDs of certain symmetric structured matrices are presented in 
\cite{dopicokoev,PelaezMoro06}.

We also know of two accurate {\em nonsymmetric\/} eigenvalue algorithms,
for totally nonnegative (TN) and for certain sign regular matrices, 
which we call TN$^J$ (see Section \ref{sec_TN} for definitions).

In the TN case, the trick is to implicitly perform an accurate 
similarity transformation to a {\em symmetric\/} tridiagonal positive 
definite matrix which is available to us in factored form. The TN 
eigenvalue problem is thus reduced to the bidiagonal SVD problem.

The sign-regular TN$^J$ matrices are similar to symmetric anti-bidiagonal 
matrices \cite{antibidiag} (i.e., the only nonzero entries are on the antidiagonal and one 
sub-antidiagonal). This similarity can be performed accurately by 
transforming implicitly an appropriate bidiagonal decomposition of the 
TN$^J$ matrix. Finally, the eigenvalues of the anti-bidiagonal matrix are 
its singular values with appropriate signs known from theory.

\begin{table}
\begin{center}

{
\begin{tabular}{|l|c|c|c|c|c|c|c|c|c|c|c|c|c|}

\hline

Type of &&& Any &\multicolumn{3}{|c|}{Gauss.
elim.}
 &&&&&&& \\
 matrix
&$\!\! \det A\!\!$& $\!\!A^{-1}\!\!$&$\!\!\!$minor$\!\!\!$
& NP & PP & CP
& \!\!RRD\!\! &  \!\!QR\!\!
& \!\!NE\!\! & \!\!$Az\!=\!b$\!\! &
\!\!SVD\!\! & \!\!EVD\!\! & \!\!Ref\!\!
\\

\hline
\hline
 &&&&&&&&&&&&&\\
\hugeX
Acyclic
& $n$
& $n^2$
& $n$
& $n^2$
& $n^2$
& $n^2$
& $n^2$
&&&
& $n^3$
& 
& \cite{DGESVD}
\\

\hline

&&&&&&&&&&&&&\\
\hugeX
DSTU
& $n^3$
& $n^5$
& $n^3$
& $n^3$
& $n^3$
& $n^3$
& $n^3$
&&&
& $n^3$
&
& \cite{DGESVD,PelaezMoro06}
\\

\hline

&&&&&&&&&&&&&\\
\hugeX
TSC
& $n$
& $n^3$
& $n$
& $n^4$
& $n^4$
& $n^4$
& $n^4$
&&&
& $n^4$
&
& \cite{DGESVD,PelaezMoro06}
\\
\hline

\hugeX Diagonally
&&&&&&&&&&&&&
\\
dominant
& $n^3$
&
& No
& $n^3$
&
& $n^3$
& $n^3$
&&&
& $n^3$
& 
& \cite{QiangYeDD1}
\\
\hline

&&&&&&&&&&&&&\cite{alfaxueye2,mmatsvd}
\\
\hugeX
M-matrices
& $n^3$
& $n^3$
& No
& $n^3$
&
& $n^3$
& $n^3$
&&&
& $n^3$
&
& \cite{ocinneide96,PenaETNA}
\\
\hline
\hugeX
Cauchy
&&&&&&&&&&&&&\\
(non-TN)& $n^2$
& $n^2$
& $n^2$
& $n^2$
& $n^3$
& $n^3$
& $n^3$
&&$n^2$&
& $n^3$
&
& \cite{boroskailatholshevsky1,demmel98}
\\
\hline

\hugeX
Vandermonde 
&&&&&&&&&&&&&\cite{bjorckpereyra,demmel98}\\
(non-TN)&  $n^2$
             &
             & No
             &
             &
             &
             &  $n^3$
&&$n^2$&
             &  $n^3$
             &
& \cite{polyvandsvd,higham90e}
\\
\hline

\hugeX
Displacement
&&&&&&&&&&&&&\\
rank one & $n^2$
&
&
&
&
&
& $n^3$
&&&
& $n^3$
&
& \cite{demmel98}
\\

\hline

\hugeX
Totally   
&&&&&&&&&&&&&\\
nonnegative
     &  $n  $
               &  $n^3$
               &  $n^3$
               &  $n^3$
               &  $n^4$
               &  $n^4$
               &  $n^3$
               &  $n^3$
               &  $0  $
               &  $n^2$
               &  $n^3$
               &  $n^3$
& \cite{koev05,koev07}
\\
\hline

&&&&&&&&&&&&&\\
 TN$^J$
               &  $n$
               &  $n^3$
               &  $n^3$
               &  $n^3$
               &  $n^4$
               &  $n^4$
               &  $n^3$
               &  $n^3$
               &  $0  $
               &  $n^2$
               &  $n^3$
               &  $n^3$
& \cite{dopicokoev}
\\

\hline

&&&&&&&&&&&&&\\
\hugeX
Toeplitz
& No
& 
& No
& No
& No
& No
& No
& No
& No
& 
& No
& No
& \cite{focm}

\\
\hline

\end{tabular}
} 
\end{center}

\caption{Existing algorithms for accurate computations with various
classes of structured matrices. Entries like $n^2$ are meant in a 
big-$\mathcal O$ sense; see Section \ref{sec_P_intro} for details.
"No" means no accurate algorithms exist without using arbitrary
precision arithmetic; see Section \ref{sec_conseq} for details.
}
\label{table1}
\end{table}


\subsection{Designing accurate algorithms for different structured classes}
\label{sec_design}

In this section we look at the particular approaches in designing accurate
algorithms for different matrix classes in order to fill the rows of
Table~\ref{table1}, explaining only a few in detail.
Each row refers to a matrix class, and each columns to a linear algebra
problem. A table entry $n^{\alpha}$ means that an accurate linear
algebra algorithm costing $\mathcal O(n^{\alpha})$ arithmetic operations 
for
the given problem and class exists. A ``No'' entry means that no accurate 
algorithm using traditional arithmetic exists, and indeed no accurate 
algorithm exists without using arbitrary precision arithmetic, in a sense
to be made precise in Section~\ref{sec_conseq}.

We begin by explaining some of the terser column headings:
``Any minor'' means that an arbitrary minor of the matrix
may be computed accurately, not just the determinant.
``Gauss. elim NP'' means Gaussian elimination with No Pivoting (GENP),
and similarly ``PP'' and ``CP'' refer to Partial Pivoting (GEPP) and
Complete Pivoting (GECP), resp. ``RRD'' is a Rank Revealing Decomposition
as described above
(frequently but not always the same as GECP). ``NE'' is Neville Elimination
\cite{gascapena92}, a variation on GENP where $L$ and $U$ are represented as
products of bidiagonal matrices (corresponding to elimination where
a multiple of row $i$ is added to row $i+1$ to create one zero entry).
$Az=b$ refers to solving $Az=b$ accurately given conditions on $b$
(alternating signs in its components).

\subsubsection{Acyclic matrices}

A matrix $A$ is called {\em acyclic\/} if its graph is; namely, the bipartite
graph with one node for each row and one node for each column and an edge
$(i,j)$ if $A_{ij}$ is nonzero. Acyclic matrices include
bidiagonal matrices (see Section~\ref{sec_BidiagonalSVD}),
and broken arrow matrices (which are nonzero only on the
diagonal and one row or one column), among exponentially many other
possibilities \cite{demmelgragg}.

Acyclic matrices are precisely the class of matrix sparsity patterns with 
the property that the Laplace expansion of each minor can have at most one 
nonzero term \cite{demmelgragg}. Thus every nonzero minor can be computed 
accurately as the product of $n$ matrix entries. Any acyclic matrix is 
also a DSTU matrix (see the following section), and so the algorithms for 
DSTU matrices may be used.



\subsubsection{DSTU (diagonal scaled totally unimodular) matrices}
\label{sec_DSTU}

A matrix $A$ is called {\em Totally Unimodular (TU)\/}
if all its minors are $0,1,$ or $-1$.
A matrix is {\em Diagonally Scaled Totally Unimodular (DSTU)\/} if it is of
the form $A = D_1ZD_2$, where $D_1$ and $D_2$ are diagonal and $Z$ is totally
unimodular.

Accurate LDU and SVD algorithms  for DSTU matrices were presented 
in \cite{demmel98} and are based on the following observation:
\begin{enumerate}
\item The Schur complement of a DSTU matrix is DSTU.
\item If at any step in the inner loop of Gaussian elimination
the subtraction
\begin{equation}
\label{eqn_innerloopGE}
a_{ij}'=a_{ij}-\frac{a_{ik}a_{kj}}{a_{kk}}
\end{equation}
has two nonzero operands, then the result $a_{ij}'$ must be exactly 0.
\end{enumerate}
In other words, to make Gaussian elimination accurate, a one-line addition
is required to test if both $a_{ij}$ and $\frac{a_{ik}a_{kj}}{a_{kk}}$
are nonzero, and to set $a_{ij}' = 0$ if they are.
Then the modified Gaussian elimination satisfies NIC,
yielding an accurate LDU decomposition.
LDU with complete pivoting yields an
accurate RRD (with $\kappa (L)$ and $\kappa (U)$ both bounded by 
$\mathcal O(n^2)$
\cite[Theorem 10.2]{DGESVD}), and an accurate RRD yields an accurate SVD
as discussed in Section~\ref{sec_RRD}.

If a DSTU matrix is symmetric, Pel\'aez and Moro derived accurate
algorithms that preserve and exploit the symmetry in their matrices
\cite{PelaezMoro06}. They also presented such {\em symmetric\/} algorithms
for TSC matrices discussed next.

DSTU matrices arise naturally in the formulation of eigenvalue
problems for Sturm-Liouville equations \cite{demmelkoev99},
and more general scalar elliptic PDE with suitable finite element
discretizations \cite{DGESVD}. We discuss this further below in
Section~\ref{sec_BeyondNIC}.

\subsubsection{TSC (total signed compound) matrices}

Let $\mathcal S$ be the set of all matrices with a given sparsity and sign
pattern. $\mathcal S$ is called {\em sign nonsingular (SNS)\/} if it contains only
square matrices, and the Laplace expansion of the determinant of each
$G\in\mathcal S$ is the sum of monomials of like-sign,
with at least one nonzero monomial.
$\mathcal S$ is called {\em total signed compound (TSC)\/} if
every square submatrix of any $G\in\mathcal S$ is either SNS, or
structurally singular (i.e., no nonzero monomials appear in its
determinant expansion). Acyclic matrix are obviously a special
case of TSC matrices, with at most one monomial appearing in each minor.

According to \cite[Lemma 7.2]{DGESVD} any minor of a TSC matrix may be
computed accurately using not more than $4n-1$ arithmetic
operations (and not counting various graph traversal operations).
With this computing the LDU decomposition of a TSC
matrix is easy. If at any step of Gaussian elimination the subtraction
in \eqref{eqn_innerloopGE}
is one of same-signed quantities, then $a_{ij}'$ is recomputed as a quotient
of minors, each of which is computed accurately as above. The total cost
could go up to $\mathcal O(n^4)$, but this
is still efficient, according to our convention.

\subsubsection{Diagonally dominant and $M$-matrices}
A matrix $A$ is called {\em (row) diagonally dominant\/}
if the sums $s_i = a_{ii} - \sum_{j \neq i} |a_{ij}|$ are nonnegative
for all rows $i$. If in addition its off-diagonal
entries $a_{ij}$ are nonpositive (so that $s_i = \sum_j a_{ij}$)
then it is called a {\em (row) diagonally dominant $M$-matrix.\/}
It turns out that these off-diagonal matrix entries and the $s_i$,
not the diagonal entries $a_{ii}$, are the right parameters for
doing accurate linear algebra with this class of matrices.
Intuitively, it is clear that the $s_i$ are the natural parameters 
since the conditions $s_i \geq 0$ define the class.

We explain how to do accurate LDU decomposition with no pivoting
or complete pivoting, in the case of a row diagonally dominant M-matrix.
Briefly, the algorithm can be organized to satisfy NIC
(see \cite{mmatsvd,ocinneide96,PenaETNA} for details).
For simplicity of notation, let the $n^2$ matrix parameters
be $b_{ij} = -a_{ij}$ and $s_i$, so all are nonnegative.
The diagonal elements, $a_{ii}$, are readily available accurately
as a sum of positive numbers:
\begin{equation}
a_{ii}=s_i+\sum_{i=1}^nb_{ij}.
\label{a_ii}
\end{equation}

The Schur complements computed using Gaussian elimination with complete
or no pivoting inherit the diagonally dominant M-matrix structure. The
parameters defining the Schur complement---the row sums (call them
$s_i'$)
and off-diagonal elements (call them $a_{ij}'=-b_{ij}'$)---are rational
functions with positive
coefficients in the $s_i$'s and $b_{ij}$'s:
$$
s_i'=s_i+\frac{b_{i1}}{a_{11}}s_1,\hskip0.2in
b_{ij}'=b_{ij}+\frac{b_{i1}}{a_{11}}b_{1j},
$$
with $a_{ii}$ given by \eqref{a_ii}.
Since the above expressions satisfy NIC, the LDU decomposition
computed using them will be accurate, as will the subsequent SVD.

Several improvements on this results have been made.
Pe\~na suggested in \cite{PenaETNA} an alternative diagonal pivoting
strategy which guarantees $L$ and $U$ to be well conditioned (as opposed
to ``well conditioned in practice'' which is what Gaussian elimination
with complete pivoting delivers).
Ye generalized this approach to symmetric diagonally dominant matrices
(removing the restriction on the signs of off-diagonal elements) \cite{QiangYeDD1,
QiangYeDD2}. It turns out that in the process of Gaussian elimination with
complete pivoting updating the $s_i$ and the diagonal entries still satisfies NIC.
However, there can be (arbitrary) cancellation in the off-diagonal entries.
Nonetheless, Ye shows that the errors in the off-diagonal entries can
be bounded in absolute value so as to be able to guarantee that $L$ and
$U$ are computed with small norm-wise errors, which is all that is
required for an RRD to in turn provide an accurate SVD.

\subsubsection{Matrices with displacement rank one}

Matrices $A$ that satisfy the Sylvester equation
$$
DA-AT=B,
$$
where $B=uv^T$ is unit rank, are said to have {\em displacement rank one.\/}
In the easiest case, when $D$ and $T$ are diagonal
($D=\diag(d_1,d_2,\ldots,d_n),\;
T=\diag(t_1,t_2,\ldots,t_n)$), $A$ is a (quasi-Cauchy)
matrix $a_{ij}=\frac{u_iv_j}{d_i-t_j}$
\cite{kailatholshevsky95, kailatholshevsky97}.

The quasi-Cauchy structure is preserved in the process of Gaussian
elimination with complete pivoting \cite{demmel98, DGESVD}. The 
explicit
formula for a determinant (or a minor) of a (quasi-)Cauchy matrix
 satisfies NIC as mentioned before.
In fact, Gaussian elimination can be made accurate still at a cost
of $\mathcal O(n^3)$ just by changing the inner loop from 
\eqref{eqn_innerloopGE} to
\[
a'_{ij} = a_{ij} \cdot \frac{(d_i - d_k)(t_k - t_j)}{(d_k - t_j)(d_i - t_k)}
\]

This is the starting point in computing the SVD of many displacement rank
one matrices. The Vandermonde matrix $V=\big[x_i^{j-1}\big]_{i,j=1}^n$ has
a displacement rank one, where $D=\diag(x_1,x_2,\ldots,x_n)$ and $T$ is
the lower shift matrix $t_{i,i-1}=1, i=1,2,\ldots,n-1, t_{1n}=1$.

Then $DA-AT=(x_1^n-1,x_2^n-1,\ldots,x_n^n-1)^T(0,0,\ldots,0,1)\equiv B$.
The matrix $T$ is circulant (and a root of unity) and is diagonalized
$T = Q \Lambda Q^*$ by
the (unitary) matrix of the DFT $Q_{ij}=\alpha^{(i-1)(j-1)}$,
where $\alpha$ is a primitive $n$th root of unity, with eigenvalues
$\Lambda_{ii} = \alpha^{(i-1)(n-1)}$.

Thus $DA-AQ\Lambda Q^*=B$, and so $D(AQ)-(AQ)\Lambda=BQ$, i.e., $AQ$ is a
quasi-Cauchy matrix (since BQ still has rank one). Now from an accurate
SVD of $AQ=U\Sigma V^*$ we automatically obtain an accurate SVD of
$A=U\Sigma (QV)^*$. But note that we need both the constant matrices
$Q$ and $\Lambda$ for this to work, which goes beyond NIC.

The same idea generalizes to other displacement rank one matrices. For
example, if $DA-AQ=B$ and $D$ and $T$ are unitarily
diagonalizable, $D=QD_1Q^*$ and $T=SD_2S^*$, then
$$
D_1 (Q_1^*AQ_2) - (Q_1^*AQ_2) D_2 = (Q_1^* u)(v^T Q_2)
$$
and $Q_1^*AQ_2$ is a quasi-Cauchy matrix. If the decompositions
$D=QD_1Q^*$ and $T=SD_2S^*$, and the products
$Q_1^* u$ and $v^T Q_2$ can be formed accurately, then from an accurate
SVD of the quasi-Cauchy matrix
$Q_1^*AQ_2=U\Sigma V^*$ we obtain an accurate SVD of $A$:
$A=(Q_1U)\Sigma (Q_2V)^*$. This approach works, e.g., for polynomial
Vandermonde matrices involving orthogonal polynomials
\cite{polyvandsvd} (see also \cite{DGESVD, demmel98,
HighamPolyVand, kailatholshevsky97}), but again requires knowing
certain constants accurately, thus going beyond NIC.

\subsubsection{Totally nonnegative and TN$^{J}$ sign regular matrices}
\label{sec_TN}

The matrices all of whose minors are nonnegative are called {\em  Totally 
Nonnegative (TN).\/} Despite this seemingly severe restriction on the minors,
TN matrices arise frequently in practice---a Vandermonde matrix with
positive and increasing nodes, the Pascal matrix, and the Hilbert matrix
are all examples of TN matrices. The first reference in the literature
(that we are aware of) for accurate matrix computations dates back to 1963
for a Vandermonde matrix with positive and increasing nodes
in an example of Kahan and Farkas \cite{KahanF63,KahanF63a,KahanF63b}.
This phenomenon
was rediscovered by Bj\"orck and Pereyra in their celebrated paper
\cite{bjorckpereyra} and later carefully analyzed and generalized
\cite{boroskailatholshevsky1, higham87BP, higham90e, marcomartinez,
demmelkoevBP, martinezpena, martinezpena2, martinezpena03}. All these
methods are based on explicit decompositions of the corresponding matrices
where all entries of the decompositions may be computed with
expressions satisfying NIC.

These ideas generalize to {\em any\/} TN matrix \cite{koev05, koev07} and are
based on a structure theorem for TN matrices
\cite{fallat01, gascapena92, gascapena96}: Any nonsingular TN matrix can 
be
decomposed as a product of nonnegative bidiagonal factors:
\begin{equation}
A=L^{(1)}L^{(2)}\cdots L^{(n-1)}DU^{(n-1)}\cdots U^{(1)}.
\label{bidiag}
\end{equation}
As mentioned before, this variation on Gaussian elimination,
called Neville elimination,
arises by eliminating all off-diagonal matrix entries
by adding a multiple of row (resp., column) $i$ to row (resp., column)
$i+1$ to zero out one entry,
and eliminating entries diagonal by diagonal, from
the outermost (with row (resp., column) multipliers stored in
$L^{(1)}$ (resp., $U^{(1)}$)) to innermost (with row (resp., column) 
multipliers
stored in $L^{(n-1)}$ (resp., $U^{(n-1)}$)).
There are exactly $n^2$ independent nonnegative parameters in the above
decomposition. They parameterize the space of {\em all\/} TN matrices.

It turns out that it is possible to perform essentially all linear algebra
on TN matrices by using only TN-preserving transformations.
In other words, given the parameterization of $A$ in (\ref{bidiag}),
it is possible to accurately compute the parameterization of
a submatrix, (unsigned) inverse, Schur complement, converse, or
product of two such matrices, all in $\mathcal O(n^3)$ time and satisfying 
NIC
\cite{koev07}. In other words, the ability to do accurate linear algebra
is ``closed'' under all these operations.
Furthermore, based on NIC,
it is possible to accurately reduce such a parameterized
matrix to bidiagonal form, enabling an accurate SVD, and to accurately
reduce it to tridiagonal form $T = BB^T$ by a similarity,
reducing the nonsymmetric eigenvalue problem to an accurate SVD \cite{koev05}.
Thus, virtually all linear algebra with TN matrices can be performed accurately.

The only remaining question is about the starting point of this
approach -- the accurate bidiagonal decompositions of the original matrix.
The entries of the bidiagonal decomposition are products of quotients of
initial minors (i.e., contiguous minors that include the first row or
column). Thus for virtually all well known TN matrices -- Pascal, Vandermonde,
Cauchy (as well as their products, Schur complements, etc.) there are
accurate formulas for their computation
\cite{boroskailatholshevsky1,koev05,martinezpena,martinezpena03}.
\ignore{
{\em Plamen: shouldn't we say something about your result for
generalized Vandermondes, that the obvious formula for a Schur function
as a positive but exponentially large sum can be reduced to polynomial
size? I'm not sure what the best statement is, or which of your papers
to refer to, for getting all the (nested) Schur functions needed
for NE.}
}

A matrix is {\em sign regular\/} \cite{gantmacher_krein} if all minors of the same order have the
same sign (but not necessarily all positive as is the case with TN
matrices).
A row- or a column-reversed TN matrix is sign regular, and the
class is such matrices is denoted TN$^J$.
Most linear algebra problems for TN$^J$ matrices follow trivially from the
corresponding TN algorithms, except for the eigenvalue algorithm
\cite{koev07}, which requires a TN$^J$-preserving transformation into a
symmetric anti-bidiagonal matrix.

We believe that the eigenvalue algorithms for TN and TN$^J$ are the first
examples of accurate eigenvalue algorithms for nonsymmetric matrices.

\subsection{Going beyond NIC (no inaccurate cancellation)}
\label{sec_BeyondNIC}

We have cited several examples where we can do more general
classes of accurate structured matrix computations by using
more general building blocks than permitted by insisting
on no inaccurate cancellation (NIC).

An accurate SVD of a Vandermonde matrix required knowing
roots of unity accurately (or more precisely, being
able to perform the operation $x - \alpha$ accurately,
where $\alpha$ is a root of unity).
More general displacement rank one problems required
similar accurate operations for constants $\alpha$
drawn from eigenvalues from a fixed sequence of
matrices, as well as the knowledge of the orthogonal
eigenvectors of these matrices.

Most interestingly, by allowing ourselves to accurately compute a given set 
of polynomials, but all of bounded numbers of terms and degrees, we can 
extend our DSTU approach from being able to accurately find eigenvalues of 
only rather simply discretized differential equations, to accurately 
compute all the eigenvalues of the scalar elliptic partial differential 
equation $\nabla \cdot (\theta \nabla u) + \lambda \rho u = 0$ on a domain 
$\Omega$ with zero Dirichlet boundary conditions, where $\theta(x)$ and 
$\rho(x)$ are scalar functions discretized on a general triangulated mesh 
in a standard way (isoperimetric finite elements on a triangulated mesh). 
In this case it is the smallest eigenvalues that are of physical interest, 
and they are accurately determined by the coefficients of the PDE. This 
result depends on a novel matrix factorization of the discretized 
differential operator in \cite{BomanHendricksonVavasis}.

It is examples such as these that encourage us to systematically
ask what expressions we can accurately evaluate, including by
allowing ourselves additional ``black boxes'' as building blocks.
This is the topic of the next section.


\section{Accurate algorithms for polynomial evaluation}
\label{sec_OlgaIoana}

In this section we give a partial answer to the question ``when can a multivariate (real or complex) polynomial be evaluated accurately?''  These results (except for Section \ref{pos_res}) have been published, with completely rigorous proofs, in \cite{focm}; we provide here intuitions and proof sketches. 

To summarize the content of this section, we give (sometimes tight) necessary and sufficient conditions for accurate multivariate polynomial evaluation over given domains. These conditions depend strongly on the type of arithmetic chosen, specifically on the type of ``basic'' operations allowed, as well as on the domain that the inputs are taken from (and also on whether the inputs belong to $\mathbb{R}^n$ or to $\mathbb{C}^n$).

Intuitively, accurate evaluation of small quantities is a more complicated issue than accurate evaluation of large quantities; thus the ``interesting'' domains, as we will see, lie arbitrarily close to or intersect the \emph{variety} of the polynomial (the set of points where the polynomial is $0$). Evaluation on domains that are not of this type (but are otherwise sufficiently well-behaved) is easy (see Section \ref{sec_PositivePolys}). Therefore, the variety plays a \emph{necessary role}.

\begin{example} To illustrate the role of the variety, we use the following example. Consider the 2-parameter family of polynomials
\[
M_{jk}(x) = j \cdot x_3 ^6 + x_1^2 \cdot x_2^2 \cdot
(j \cdot x_1^2 + j \cdot x_2^2 - k \cdot x_3^2)~,
\]
where $j$ and $k$ are positive integers, and the domain of evaluation is $\mathbb{R}^3$. Assume that 
we allow only addition, subtraction and multiplication
of two arguments as basic arithmetic operations, along
with comparisons and branching.

When $k/j < 3$, $M_{jk}(x)$ is {\em positive definite,\/} i.e.,
zero only at the origin and positive elsewhere.  This will mean that
$M_{jk}(x)$ is easy to evaluate accurately using a simple method
discussed in Section \ref{sec_PositivePolys}.

When $k/j > 3$, then we will show that $M_{jk}(x)$
cannot be evaluated accurately by {\em any\/} algorithm
using only addition, subtraction and multiplication of
two arguments. This will follow from a simple necessary
condition on the real variety $V_{\Rr}(M_{jk})$,
the set of real $x$ where $M_{jk}(x)=0$, see Theorem~\ref{conj}.

When $k/j=3$, i.e., on the boundary between the above two cases,
$M_{jk}(x)$ is a multiple of the Motzkin polynomial
\cite{reznick2000}. The real variety $V_{\Rr}(M_{jk}) = \{x: |x_1| = |x_2| = |x_3| \}$ of this polynomial
satisfies the necessary condition of Theorem~\ref{conj}, and the simplest
accurate algorithm to evaluate it that we know of has 8 cases depending on
the relative values of $|x_i \pm x_j|$. For example, on the branch defined by the inequalities $x_1 - x_3| \leq |x_1+x_3| \wedge |x_2 - x_3| \leq x_2+x_3|$, the algorithm evaluates $p$ using the non-obvious formula
\begin{eqnarray*}
 p(x_1, x_2, x_3) & = & x_3^4 \cdot [4((x_1-x_3)^2 + (x_2-x_3)^2 + (x_1-x_3)(x_2-x_3))] \\
         &   &  ~+~ x_3^3 \cdot [2(2(x_1-x_3)^3 + 5(x_2-x_3)(x_1-x_3)^2  \\       &   &   \hspace*{.5in} +~ 5(x_2-x_3)^2(x_1-x_3) + 2(x_2-x_3)^3)] \\
         &   & ~+~ x_3^2 \cdot [(x_1-x_3)^4 + 8(x_2-x_3)(x_1-x_3)^3 + (x_2-x_3)^4 \\         &   &  \hspace*{.5cm}  +~ 9(x_2-x_3)^2(x_1-x_3)^2 +  8(x_2-x_3)^3 (x_1-x_3)] \\
         &   & ~+~ x_3 \cdot [2(x_2-x_3)(x_1-x_3)((x_1-x_3)^3 + (x_2-x_3)^3 \\
        &   &  \hspace*{.5in} +~2(x_2-x_3)(x_1-x_3)^2 + 2(x_2-x_3)^2(x_1-x_3)] \\
        &   & ~+~ (x_2-x_3)^2(x_1-x_3)^2((x_1-x_3)^2 + (x_2-x_3)^2)~.
\end{eqnarray*}

In contrast to the real case, when the domain is $\C^3$, Theorem~\ref{conj} will show that $M_{jk}(x)$ cannot be
accurately evaluated using only addition, subtraction and
multiplication.
\end{example}

The necessary conditions we obtain for accurate evaluability depend only on the variety of $p(x)$, but the variety alone is not always enough.

\begin{example} Consider the irreducible, homogeneous, degree $2d$, real polynomial
\[
p(x) = (x_1^{2d} + x_2^{2d}) + (x_1^2 + x_2^2)(q(x_3,...,x_n))^2~,
\]
where $q(\cdot)$ is homogeneous of degree $d-1$.
The variety $V(p) = \{ x_1=x_2=0 \}$ satisfies the necessary condition for
accurate evaluability, but near $V(p)$ the polynomial
$p(x)$ is ``dominated'' by
$(x_1^2 + x_2^2)(q(x_3,...,x_n))^2$, so accurate
evaluability of $p(x)$ depends on the accurate evaluability of $q(\cdot)$.

We may now apply the same principle to
$q( \cdot )$, etc., thus creating a decision tree of polynomials.
Rather than a characterizing theorem, one might expect therefore that, in many cases, the answer can only be given by a recursive decision procedure, expanding $p(x)$ near the components of its variety and so on. We discuss this more in Section \ref{sec_traditional}. 
\end{example} 

The rest of Section \ref{sec_OlgaIoana} is structured as follows. 
In Section \ref{sec_models}, we formalize the type of algorithms we are interested in. 
Section \ref{sec_PositivePolys} makes rigorous the intuition that accurate evaluation ``far from the variety'' is possible. 
Section \ref{sec_traditional} considers the traditional model of arithmetic, on ``well-behaved'' domains similar to the ones chosen for the algorithms of Section \ref{sec_Plamen}. This model has three basic operations: $+, -, \times$, and allows for exact negation. While not sufficient for the accurate evaluation \emph{everywhere} of even simple polynomial expressions like $x+y+z$, the traditional model is simple enough to allow us to give a characterization of accurately evaluable \emph{complex} polynomials, as well as (generally distinct) necessary and sufficient conditions for accurate evaluability of real polynomials (sometimes these conditions are identical, and offer a complete characterization). In addition, for the real case, we show current progress toward constructing a decision procedure for accurate evaluability of real polynomials.
Section \ref{sec_extended} expands the practical scope of our analysis, since concluding that a computation is "impossible" is not the end of the story;  instead, this begs the question of what additional computational building blocks would be needed to make it possible? For example, current computers often have a "fused multiply-add" instruction $x + y \cdot z$ that computes the answer with one rounding error, and there are software libraries that provide collections of accurately implemented polynomials needed for certain applications, e.g., computational geometry \cite{shewchuk}. Given any such a collection of what we will call ``black-box'' operations (about which we assume only a small relative error), we will ask how much larger a set of polynomials can be evaluated accurately.

Finally, Section \ref{sec_conseq} discusses the implications of these results. Firstly, they shed some light on the existence of accurate algorithms for linear algebra operations like the ones described in Section \ref{sec_Plamen}: each such algorithm satisfies NIC (see Section \ref{sec_intro}, and thus also satisfies the necessary condition for accurate evaluability presented in Theorem \ref{conj}). The apparently unrelated classes of structured matrices for which efficient and accurate linear algebra algorithms exist share a common underlying algebraic structure. Also, there may be other structured matrix classes sharing this property and for which accurate algorithms could be built.
Secondly, our results show that some expressions or classes of problems \emph{cannot} be accurately evaluated, even with an arbitrary set of bounded-degree black-box operations at our disposal. The practical implication of this is that, for certain types of problems, the use of arbitrarily high precision is necessary (see Section \ref{sec_OtherModels}).  Lastly, but perhaps most importantly, our results lay down a path toward the ultimate goal: a decision procedure (or ``compiler'') which, given as inputs a polynomial $p$, a domain $\mathcal{D}$, and (perhaps) a set of black-box operations, either produces an accurate algorithm for the evaluation of $p$ on $\mathcal{D}$ (including how to choose the machine precision $\epsilon$ for the desired relative error $\eta$, see Section \ref{sec_models}), or exhibits a ``minimal'' set of black-box operations that are still needed. 

\subsection{Formal statement and models of algorithms} \label{sec_models}

We formalize here both the problem and the models of algorithms we will use. We introduce the notation $p_{comp}(x, \delta )$ for the output of the algorithm, and $\delta = ( \delta_1  , \delta_2, ... , \delta_k )$ for the vector of rounding errors. 

For example, consider the algorithm that computes $p(x) = x_1+x_2+x_3$ by performing two additions: first adds $x_1$ to $x_2$, then adds the result to $x_3$. If the first and second additions introduce the relative errors $\delta_1$, respectively $\delta_2$, we obtain that, for this algorithm,
\begin{eqnarray}
p_{comp}(x, \delta) & = & \left( (x_1+x_2) (1+\delta_1) + x_3 \right) (1+ \delta_2) \nonumber \\
& = & (x_1+x_2+x_3)(1+ \delta_2) + (x_1+x_3) \delta_1 (1+\delta_2)~. \label{gugu}
\end{eqnarray}

We give below a formal description of the algorithms we consider.
For more in-depth discussion of these assumptions and comparisons with other models
of computations, see  Section~\ref{sec_OtherModels}.

\begin{definition} All algorithms considered in this section will satisfy the following constraints.
\begin{enumerate}

\item
The inputs $x$ are given exactly, rather than approximately.

\item
The algorithm always computes the output $p_{comp}(x,\delta)$ in finitely many steps and, moreover, computes  the exact value
of $p(x)$ when all rounding errors $\delta = 0$.  This constraint 
excludes iterative algorithms which might produce an approximate value
of $p(x)$ even when $\delta = 0$. Some of the reasons for this choice can be found in Section \ref{sec_Tools}.

\item
The basic arithmetic operations 
beyond the traditional addition, subtraction and multiplication, if any, must be given explicitly. We refer to the case when additional polynomial
operations are included as \emph{extended arithmetic}.
Constants are available to our algorithms only in the extended model and are also given explicitly.

\item
We consider algorithms both with and without comparisons and
branching, since this choice may change the set of polynomials
that we can accurately evaluate. In the branching case, note that $p_{comp}(x, \delta)$ will actually be piecewise polynomial. 


\item
If the computed value of an operation depends only on the
values of its operands, i.e., if the same operands $x$ and $y$
of $op(x,y)$ always yield the same $\delta$ in
$rnd(op(x,y)) = op(x,y) \cdot (1+ \delta)$, then we
call our model {\em deterministic,\/} else it is
{\em nondeterministic.\/} One can show that comparisons and branching
let a nondeterministic machine simulate a deterministic one,
and subsequently restrict our investigation to the easier
nondeterministic model.
\end{enumerate}
\end{definition}

Finally, we must formalize what type of domains we consider.
Though, in principle, any semialgebraic set $\cal D$ could be examined, for simplicity we consider open domains $\cal D$, especially ${\cal D} = \Rr^n$ or ${\cal D} = \C^n$. 
We can now give the formal definition of accuracy.

\begin{definition}
We say that $p_{comp}(x, \delta)$ is an {\em accurate algorithm}
for the evaluation of $p(x)$ for $x \in {\cal D}$ if 

\vspace{.15cm}

\hspace*{-0.10in}    $\forall \; 0 < \eta < 1 \; \; \; \; \;$ ...
                 for any $\eta$ = desired relative error \\
\hspace*{0.23in} $\exists \; 0 < \epsilon < 1 \; \; \; \;$ ...
                 there is an $\epsilon$ = machine precision \\
\hspace*{0.46in} $\forall \; x \in {\cal D} \; \; \; \; \;$ ...
                 so that for all $x$ in the domain \\
\hspace*{0.69in} $\forall \; |\delta_i| \leq \epsilon \; \; \;$ ...
                 and for all rounding errors bounded by $\epsilon$ \\
\hspace*{0.92in} $|p_{comp}(x,\delta) - p(x)| \leq \eta \cdot |p(x)|$ ...
                 the relative error is at most $\eta$.
\\
\end{definition}

Note that the algorithm proposed above, which produces the $p_{comp}$ given in \eqref{gugu} for the evaluation of $x_1+x_2+x_3$ is not an accurate algorithm (consider the case when $x_1+x_2 = -x_3$). This is not accidental (see Theorem \ref{conj}).


Given an algorithm producing a polynomial $p_{comp}$, the problem of deciding whether it is accurate \emph{is} a Tarski-decidable problem \cite{renegar92,tarski_book}. What is unclear if whether the \emph{existence} of an accurate algorithm for a given polynomial and domain is a Tarski-decidable problem, since we see no way to express ``there exists an algorithm'' in the required format.

\subsection{The bounded from below case (empty variety)}
\label{sec_PositivePolys}
 
We consider the simpler case where the polynomial $p(x)$
to be evaluated is bounded (in absolute value) above and below, in an appropriate manner, on the domain $\cal D$ (this is what we referred to previously as ``far from the variety'', i.e., the set where the polynomial is $0$).
If the domain $\cal D$ is compact, we give here, with proof, the following theorem. (We let $\bar{\cal D}$ denote the closure of $\cal D$.)

\begin{theorem}
\label{thm_positive_compact}
Let $p_{comp} (x,\delta)$ be {\em any} algorithm computing
$p(x)$ satisfying $p_{comp}(x,0) = p(x)$,
i.e., it computes the right value in the absence of rounding error.
Let $p_{min} := \inf_{x \in \bar{\cal D}} |p(x)|$.
Suppose $\bar{\cal D}$ is compact and $p_{min} > 0$.
Then $p_{comp}(x,\delta)$ is an accurate algorithm
for $p(x)$ on $\cal D$.
\end{theorem}

\begin{proof}
Since the relative error on $\cal D$ is
$$|p_{comp}(x,\delta) - p(x)|/|p(x)| \leq |p_{comp}(x,\delta) - p(x)|/p_{min}~,$$
it suffices to show that the right hand side numerator approaches 0 uniformly
as $\delta \rightarrow 0$. This follows by writing the value of
$p_{comp}(x,\delta)$ along any branch of the algorithm as
$$p_{comp}(x,\delta) = p(x) + \sum_{\alpha > 0} p_{\alpha}(x) \delta^{\alpha}~,$$
where $\alpha > 0$ is a multi-index with at least one component exceeding 0.
By compactness of $\bar{\cal D}$, all $p_{\alpha}$ are bounded on $\bar{\cal D}$, and thus there exists some constant $C>0 $ such that
$$|\sum_{\alpha > 0} p_{\alpha}(x) \delta^{\alpha} | \leq C
\sum_{\alpha > 0} |\delta|^{\alpha}~.$$ The right hand side goes to 0 uniformly as the upper bound $\epsilon$ on
each $|\delta_i|$ goes to zero.
\end{proof}

What about domains that are not compact, e.g., not bounded? The proof above points to some of the issues that may occur: ratios $p_{\alpha}(x)/p(x)$ could become unbounded, even though $p_{min}>0$. Another way to see that requiring $p_{min}>0$ is not enough is to consider the polynomial $$p(x)= 1+ (x_1+x_2+x_3)^2~.$$ To evaluate this polynomial accurately, intuitively, one needs to evaluate $(x_1+x_2+x_3)^2$ accurately, once it is sufficiently large. If one uses only addition, subtraction, and multiplication, this is not possible. (These considerations will be made explicit in Section \ref{suf_r}.)

There are, however, cases in which unboundedness is not an impediment. Consider the case of a homogeneous polynomial $p(x)$, to be evaluated on a homogeneous domain $\mathcal{D}$ (i.e., a domain with the property that $x \in \mathcal{D}$ implies $\gamma x \in \mathcal{D}$, for any scalar $\gamma$). Due to the homogeneity of $p$, we can then restrict our analysis to $\mathcal{D} \cap S^{n-1}$ (the unit ball in $\mathbb{R}^n$), or $\mathcal{D} \cap S^{2n-1}$ (the unit ball in $\mathbb{C}^n$). On such domains we can use a compactness argument, as we did before, to obtain:

\begin{theorem}
\label{thm_positive_homo}
Let $p(x)$ be a homogeneous polynomial, let $\cal D$
be a homogeneous domain, and let $S$ denote the unit
ball in $\Rr^n$ (or $\C^n$).
Let
\[
p_{min,homo} \equiv \inf_{x \in \bar{\cal D} \cap S} |p(x)| \;
\]
Then $p(x)$ can be evaluated accurately if $p_{min,homo} > 0$.
\end{theorem}

A simple, Horner-like scheme that provides an accurate 
$p_{comp}(x,\delta)$ in this case is given in \cite{focm}, along with a proof.

\subsection{Traditional arithmetic} \label{sec_traditional}

In this section we consider the basic or traditional arithmetic over the real or complex fields, with the three basic operations $\{+, -, \times\}$, to which we add negation. The model of arithmetic is governed by the laws in Section \ref{sec_models}, and has also been described in Section \ref{sec_Plamen}. We remind the reader that this arithmetic model \emph{does not allow} the use of constants.

Section \ref{class_nec} describes the necessary condition for accurate evaluability over both real and complex domains. Sections \ref{suf_c}, respectively \ref{suf_r} deal with sufficient conditions for accurate evaluability over $\mathbb{C}^{n}$, respectively $\mathbb{R}^n$. We show that the necessary and sufficient conditions for accurate evaluation coincide in the complex case, in Section \ref{suf_c}. Section \ref{suf_r} also describes progress toward understanding how to construct a decision procedure in the real case.

Throughout this section, we will make use of the following definition of allowability.

\begin{definition} Let $p$ be a polynomial over $\mathbb{R}^n$ or $\mathbb{C}^n$, with variety  $V(p)\eqbd \{ x~:~p(x)=0\}$. 
We call $V(p)$ \emph{allowable} if it can be represented as a union of intersections of hyperplanes of the form
\begin{eqnarray} \label{unu}
& 1.&  Z_i =\{x~:~ x_i ~=~ 0\}~, \\
\label{doi}
& 2.&  S_{ij} = \{x~:~ x_i+x_j ~=~ 0\}~, \\
\label{trei}
& 3.&  D_{ij} = \{x~:~ x_i - x_j ~=~ 0\}~.
\end{eqnarray}
If $V(p)$ is not allowable, we call it unallowable.
\end{definition}

The word ``allowable'' in the definition above is used because, as we will see, polynomials with ``unallowable'' varieties do not allow for the existence of accurate evaluation algorithms.

For a polynomial $p$, having an allowable variety $V(p)$ is obviously a Tarski-decidable property (following \cite{tarski_book}), since the number of unions of intersections of hyperplanes \eqref{unu}-\eqref{trei} is finite.

\subsubsection{Necessity: real and complex} \label{class_nec}

All the statements, proofs, and proof sketches in this section work equally well for both the real and the complex case, and thus we will treat them together.

Throughout this section we will denote the variable space by $\S \in \{\Rr^n, \C^n\}$.

To state and explain the main result of this section, we need to introduce some additional notions and notation.

\begin{definition} \label{general_p} Given a polynomial $p$ over $\S$ with unallowable variety $V(p)$, consider 
all sets $W$ that are finite intersections of allowable hyperplanes defined by \eqref{unu}, 
\eqref{doi}, \eqref{trei}, and subtract from $V(p)$ all those $W$ for which $W \subset V(p)$. 
We call the remaining subset of the variety {\em points in general position\/} and denote it by $G(p)$.
\end{definition}

If $V(p)$ is not allowable, then from Definition \ref{general_p} it follows that $G(p)\neq \emptyset$. 

\begin{definition} \label{allowance} 
Given $x \in \S$, define the set $\allow(x)$ as the intersection of
all allowable planes going through $x$:
$$ \allow(x)\eqbd \left( \cap_{x\in Z_i} Z_i\right)  \cap \left( \cap_{x\in S_{ij}} S_{ij}\right) 
 \cap \left( \cap_{x\in D_{ij}} D_{ij} \right) , $$
with the understanding that 
$$ \allow(x)\eqbd \mathcal{S} \qquad {\rm whenever} \qquad x\notin Z_i, \; S_{ij}, \; D_{ij}
\quad \hbox{\rm for all} \quad i,j.$$
Note that $\allow(x)$ is a linear subspace of $S$.
\end{definition}

In general, we are interested in the sets $\allow(x)$ primarily when $x\in G(p)$. For each such $x$, $ \allow(x)\not\subseteq V(p)$, 
which follows directly from the definition of $G(p)$.

We can now state the main result of this section, which is a necessity condition for the evaluability of polynomials over domains. In the following, we denote by $\overline{\Int({\mathcal{D})}}$ the closure of the interior of the domain $\mathcal{D}$.

\begin{theorem} \label{conj}
Let $p$ be a polynomial over a domain $\mathcal{D} \in \S$, such that $\mathcal{D} = \overline{\Int({\mathcal{D})}}$. Let $G(p)$ be the set of points in general position on the variety $V(p)$. If $\Int(\mathcal{D}) \cap G(p) \neq \emptyset$, then $p$ is not accurately evaluable on $\mathcal{D}$.
\end{theorem}

With a little more work one can see that ``failures'' are not rare. More precisely, in the same circumstances as above, any algorithm attempting to compute $p$ accurately on $\mathcal{D}$ will fail to do so consistently on a set of positive measure.

\begin{corollary} \label{in_sfirsit} Let $p$ and $\mathcal{D}$ as before, $x \in \Int(\mathcal{D}) \cap G(p)$, $\epsilon>0, ~1>\eta>0$, and $p_{comp}(\cdot, \delta)$ be the result of an algorithm attempting to compute $p$ on $\mathcal{D}$ with error vector $\delta$. Then there exists a set $\Delta_x$ {\sl arbitrarily close} to $x$ and a set $\Delta_{\delta}$ of positive measure in $H_{\epsilon} := \{\delta~:~ |\delta_i| \leq \epsilon \}$  such that $|p_{comp} - p|/|p| >\eta$ 
 when computed at any point $y \in \Delta_x$ using any vector of relative errors $\delta \in \Delta_{\delta}$.
\end{corollary}

\vspace{.25cm}

For the benefit of the reader we give here a sketch of the proof of Theorem \ref{conj} in an informal style. Details and rigorous statements can be found in \cite{focm}.

\vspace{.25cm}

\begin{proof}{Theorem \ref{conj}} The essential idea is to consider under what kind of circumstances can an algorithm in which every non-trivial operation introduces errors actually produce a perfect $0$. Note that, by definition, for an algorithm to be accurate, it must compute $p(x)$ \emph{exactly} when $x \in V(p)$, and it cannot output $0$ for any $x \notin V(p)$.

For starters, think of the algorithm as in~\cite{AHU}--as a directed acyclic graph (DAG) with input, computational, branching, and output nodes. Every computational node has two inputs (which may both come from a single other computational node). All computational nodes are labeled by $(op(\cdot), \delta_i)$ with $op(\cdot)$ representing the operation that takes place at that node. It means that at each node, the algorithm takes in two inputs, executes the operation, and multiplies the result by $(1+\delta_i)$. Finally, for every branch of the algorithm, there is a single destination node, with one input and no output, whose input value is the result of the algorithm.

For simplicity, in this sketch we only consider non-branching algorithms.

Assume that $x \in G(p)$ is fixed, and let us examine the algorithm as a function of the error variables $\delta$. Some computational nodes in this DAG might do ``trivial'' work (work that, given the input $x$, outputs $0$ for all choices of variables $\delta$). For example, such a node might receive input from a single computational node, subtract it from itself, and thus output $0$. Note that multiplication nodes cannot produce a $0$ unless they receive a $0$ as an input.

For all non-trivial computation nodes, the output result is a polynomial of $\delta$ (and thus it will only vanish on a set of $\delta$s of measure $0$). 

As such, for any $x \in G(p)$, there will be a positive measure set $\Delta$ of $\delta$s for which non-trivial nodes will not output $0$. Let us now choose some $\delta$ in this set and then look at the computational output node. Since we assume that the algorithm is accurate, the output node must be $0$, therefore the output node must be of ``trivial'' type. Let us track back zeros in the computation, marking the nodes where such zeros appear and propagate from. In other words, backward-reconstruct paths of zeros that lead to the output of the computation. 

Zeros propagate forward by multiplication, or by the addition/subtraction of identical quantities; but how do the \emph{first} zeros on such paths (from the perspective of the computation) get created? A quick analysis shows that there are only three possibilities: either they are sources (zero as an input), or come from nodes corresponding to the trivial operation of subtracting an input from itself ($q(\delta)-q(\delta)$, since the node that computed this input must have been non-trivial), or they correspond to the addition or subtraction of two equal source inputs ($x_i = x_j$ or $x_i = - x_j$). 

We illustrate these possibilities in Figure~\ref{figur2} below. The white nodes are ``trivial'' nodes, labeled with the operation executed there and the error variable; for clarity, we dropped the indices on the variables $\delta_i$, and chosen not to represent certain parts of the graph. The gray nodes are  non-trivial nodes. Arrows are labeled with the value they carry. Rectangles represent source nodes, and the triangle is the final output node.

\begin{figure}
\begin{center}
\epsfig{figure=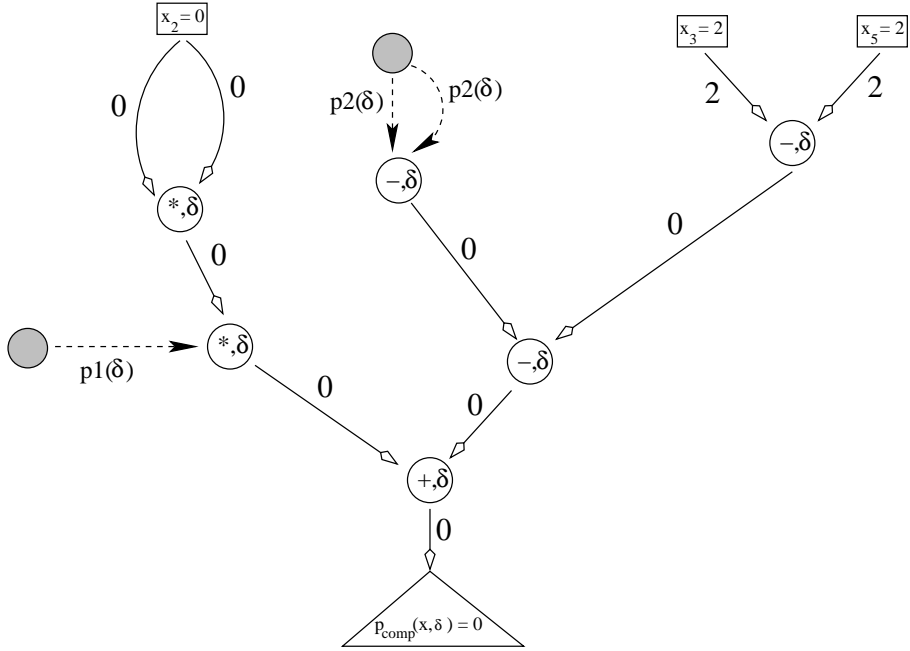, width=12cm}
\caption{The three ways to produce zeros.}
\label{figur2}
\end{center}
\end{figure}


The key observation is that \emph{all of these zeros would be preserved if we replaced $x$ with any $y \in \allow(x)$}. In other words, if the algorithm outputs $p_{comp}(x, \delta) = 0$, for some $\delta \in \Delta$, then it will also output $p_{comp}(y, \delta) = 0$, for \emph{all} $\delta \in \Delta$, and \emph{all} $y \in \allow(x)$. 

For example, assume that the polynomial in Figure \ref{figur2} is $$p(x) = (x_1+x_4+x_6)^2 +x_2^4 +(x_3-x_5)^2~,$$ with unallowable variety $V(p) = \{x_1+x_4+x_6=0\} \cap \{x_2 = 0 \} \cap \{x_3=x_5\}$, and that we want to compute $p$ at $x = (1, 0, 2, 3, 2, -4) \in G(p)$. Then the result of the computation would be correct: $p_{comp}(x, \delta) = 0$. However, this algorithm would also output $p_{comp}(y, \delta) = 0$ for the point $y = (1, 0, 2, 3, 2, 4)$, which is in $\allow(x) = \{x_2 = 0 \} \cap \{x_3=x_5\}$, but not in $V(p)$, since $p(y) = 16$.

Since $x \in G(p)$, $\allow(x) \notin V(p)$, and thus the algorithm obtains $0$ on points not in the variety, hence it fails. 
\end{proof}

\subsubsection{Sufficiency: the complex case} \label{suf_c}

Suppose we now restrict input values to be complex numbers and use the same 
algorithm types and the notion of accurate evaluability from the previous 
sections. By Theorem \ref{conj}, for a polynomial $p$ of $n$ complex 
variables to be accurately evaluable over $\C^n$ it is necessary that 
its variety $V(p)\eqbd 
\{z\in \C^n : p(z)=0 \}$ be allowable.

We give and explain here a result that shows that this condition is also sufficient. This characterization is possible in the complex polynomial case because complex varieties are (pun intended) much simpler than real ones. In particular, Theorem \ref{dimensions} has no correspondent for real varieties, and therefore we cannot prove anything close to Theorem \ref{sufficiency_c} for the real polynomial case.

\begin{theorem} \label{sufficiency_c}
Let $p: \C^n \to \C$ be a polynomial with integer
coefficients and zero constant term. Then $p$ is accurately
evaluable on $\mathcal{D} = \mathbb{C}^n$ if and only if the variety $V(p)$ is allowable.
\end{theorem}

To prove this, we first investigate allowable complex varieties.
We start by recalling a basic fact about complex polynomial varieties (Theorem \ref{dimensions}), which can for example be deduced from Theorem 3.7.4 in \cite[page 53]{T}. 
Let $V$ denote any complex variety. To say that $\dim_\C(V)=k$
means that, for each $z\in V$ and each $\delta>0$, there exists 
$w\in V\cap B(z,\delta)$ such that $w$ has a $V$-neighborhood that
is homeomorphic to a real $2k$-dimensional ball.

\begin{theorem} \label{dimensions} Let $p$ be a non-constant polynomial over $\C^n$. 
Then $$\dim_\C (V(p))=n-1.$$
\end{theorem}

\begin{corollary} \label{no_intrs} 
Let $p: \C^n \to \C$ be a non-constant polynomial 
whose variety $V(p)$ is allowable. Then $V(p)$ is a union of allowable 
hyperplanes.  
\end{corollary}

\begin{proof} Since $V(p)$ is allowable, let $V(p)=\cup_{j} S_j$ be the (minimal) way to write $V(p)$ as an irredundant union of irredundant intersections of hyperplanes. Assume that, for some $j_0$, $S_{j_0}$ is not a hyperplane but an (irredundant) intersection of hyperplanes.
Let $z\in {S_{j_0}\setminus \cup_{j\neq j_0} S_j}$. Then, for some
$\delta>0$, $B(z,\delta)\cap V(p)\subset S_{j_0}$. Since $\dim_C(S_{j_0})<n-1$,
no point in $B(z,\delta)\cap V(p)$ has a $V(p)$-neighborhood that is 
homeomorphic to a real $2(n-1)$-dimensional ball. Contradiction.
\end{proof}
 
\begin{corollary} \label{factors} If $p: \C^n \to \C$ is a 
polynomial whose variety $V(p)$ is allowable, then it is a product 
$p=c \prod_j p_j$, where each $p_j$ is a power of $x_i$, $(x_i-x_j)$, or $(x_i +x_j)$.
\end{corollary}

\begin{proof} By Corollary~\ref{no_intrs}, the variety $V(p)$ is
an irredundant union of allowable hyperplanes. 

Choose a hyperplane $H$ in that union. If $H = Z_{j_0}$ for some $J_0$, expand $p$ into a Taylor series in $x_{j_0}$. If $H = D_{i_0 j_0}$ (or $H = S_{i_0 j_0}$)  for some $i_0$, $j_0$, expand $p$ into a Taylor series in $(x_{i_0} - x_{j_0})$ (or  $(x_{i_0} + x_{j_0})$). In this case, the zeroth coefficient of $p$ in the expansion must be the zero polynomial in $x_j$, $j\neq j_0$ (or $ j \notin \{i_0, j_0\}$). 
Hence there is a $k$ such that $p(x)= x_{j_0}^{k} ~\widetilde{p}(x)$ in the first case, or $p(x) = (x_{i_0} \pm x_{j_0})^k ~ \widetilde{p}(x)$ in the second (third) one. In any case, we choose $k$ maximal, so that $V(\widetilde{p})$ does not include $H$.

It is easy to see that the variety $V(\widetilde{p})$ must include $V(p) \setminus H$ (the union of all the other hyperplanes) -- whose dimension is $n-1$. Moreover, $V(\widetilde{p})$ (by Theorem \ref{dimensions}) has dimension $n-1$ and, by the maximality of $k$, does not include $H$. 

If $V(\widetilde{p}) \cap H \eqbd H'$ were non-empty, it would follow that dim$(H')~\leq~n-2$ (since it is included in the hyperplane $H$, and strictly smaller than $H$). This would contradict Theorem \ref{dimensions}, which states that dim$(V(\widetilde{p})) = n-1$. Therefore it must be that $V(\widetilde{p}) \cap H = \emptyset $, and thus $V(\widetilde{p})$ must equal $V(p) \setminus H$, the union of a smaller number of allowable hyperplanes. 

Proceed inductively by factoring $\widetilde{p}$ in the same fashion. 
\end{proof}

The crucial point in the proof above is that the $V(\widetilde{p}) \cap H$ \emph{must be $\emptyset$}, due to Theorem \ref{dimensions}. The same argument would break, in the real case; to illustrate this, consider the polynomial $p(x_1, x_2, x_3) = x_1^4 + x_1^2 (x_2+x_3)^2$. The variety is $V(p) = \{x_1 = 0 \}$, of dimension $2$, but after factoring out $x_1^2$, the variety of the remaining polynomial, $\widetilde{p} = x_1^2 + (x_2+x_3)^2$, is $\{x_1=0\} \cup \{x_2+x_3=0\}$ -- which has dimension $1$.
We can now prove Theorem \ref{sufficiency_c}.

{\em Proof of Theorem \sl{\ref{sufficiency_c}}.}   
By Corollary~\ref{factors},  $ p=c\prod_j p_j $, with each $p_j$ a power of $x_k$ or $(x_k \pm x_l)$. It also follows that $c$ must be an integer since all coefficients of $p$ are integers.
Since each of the factors is accurately evaluable, and we can get any integer constant $c$ in front of $p$ by repeated addition (followed, if need be, by negation), which are again accurate operations, the algorithm that forms their product and then adds/negates to obtain $c$ evaluates $p$ accurately. 
\hskip 0.2cm $\Box$

Theorem {\sl \ref{sufficiency_c}} implies that only homogeneous polynomials are accurately evaluable over $\mathbb{C}^n$.

\subsubsection{Sufficiency: toward a decision procedure for the real case} \label{suf_r}

In this section we relate the accurate evaluability of a polynomial to the accurate evaluability of its ``dominant terms'', and explore a possible avenue toward a decision procedure to establish the former via a recursive/inductive procedure based on the latter. 

We consider only homogeneous polynomials, for reasons outlined in Section \ref{sec_PositivePolys}, and we also consider separately the branching and non-branching cases. Most of the section is devoted to non-branching algorithms, but we do need branching for our statements at the end; we keep the reader informed of all changes in the assumptions.

To accurately compute a homogeneous polynomial of degree $d$ using a non-branching algorithm, one needs to use a homogenous algorithm, described by the following definition and lemma, to be used later in Section \ref{sec_prune}.

\begin{definition} \label{homalg}
We call an algorithm $p_{comp}(x,\delta)$ with error set $\delta$ for computing $p(x)$
{\em homogeneous of degree $d$\/} if 
\begin{enumerate} 
\item \label{unu1} the final output is of degree $d$ in $x$;
\item \label{doi2} no output of a computational node exceeds degree $d$ in $x$;
\item \label{trei3} the output of every computational node is homogeneous in $x$.
\end{enumerate}
\end{definition}

\begin{lemma} \label{1} If $p(x)$ is a homogeneous polynomial of degree $d$ and 
if a non-branching algorithm  evaluates $p(x)$ accurately by computing 
$p_{comp}(x,\delta)$, the algorithm must itself be homogeneous of degree $d$.
\end{lemma}

The proof involves a combination of expressing the relative errors \noindent $|p_{comp}(x, \delta) - p(x)|/|p(x)|$ as in the proof of Theorem \ref{thm_positive_compact}, and an analysis of the algorithm as a DAG, as in Section \ref{class_nec}.

Due to the complexity of the issues, the rest of this section is subdivided into four parts: \begin{itemize} 
 \item Section \ref{dom} makes rigorous the notion of dominance and explains how to find the dominant terms
 by using various simple linear changes of variables. 
\item In Section \ref{sec_prune}, we explain how to ``prune'' an algorithm to manufacture an algorithm that 
evaluates one of its dominant terms, and we establish that accurate evaluation of the dominant terms 
identified in Section~\ref{dom} is \emph{necessary} for the accurate evaluation of the polynomial. 
\item Section \ref{sec_suff} establishes that accurate evaluation of a special set of dominant terms, 
together with the slices of space where they dominate, is sufficient for accurate evaluation of the 
polynomial. \item Finally, Section \ref{obst} discusses obstacles to a complete inductive procedure. 
\end{itemize}

\subsubsection{Dominance} \label{dom}

We now describe what we mean by ``dominant terms'' of the polynomial. Given an allowable variety $V(P)$, we fix an irreducible component of $V(p)$. Any such component is described by linear allowable constraints. We note (see \cite{focm}) that any
given component of $V (p)$ can be put into the form $x_1 = x_2 = ... = x_k = 0 $ using what we call
a standard change of variables; standard changes of variables are linear transformations
of the variables, which are intuitively simple, but whose exact combinatorial definition is long and we choose to leave it out.

After a standard change of variables, we look at the component $x_1 = x_2 = ... = x_k = 0$. We can assume that the polynomial $p(x)$ can be written
 (almost following MATLAB notation) as $$p(x) = \sum_{\lambda \in \Lambda} c_{\lambda} x_{[1:k]}^{\lambda} q_{\lambda}(x_{[k+1:n]})~,$$ where we write $x_{[1:k]} :=(x_1, . . . , x_k), x_{[k+1:n]} :=(x_{k+1}, . . . , x_n)$.
Also, we let $\Lambda$ be the set of all multi-indices $\lambda :=(\lambda_1, . . . , \lambda_k)$ appearing above.

To determine all dominant terms associated with the component $x_1 = x_2 =...= x_k =0$,
consider the Newton polytope $P$ of the polynomial $p$ with respect to the 
variables $x_1$ through $x_k$ only,  i.e., the convex hull of the exponent vectors $\lambda \in \Lambda$ (see, e.g.,~\cite[p.~71]{MS}). 
Next, consider the normal fan $N(P)$ of $P$ (see~\cite[pp.~192--193]{Z}) consisting of the cones of all row vectors $\eta$ whose dot products with $x\in P$ are maximal for $x$ on a fixed face of $P$. That means that for every nonempty face $F$ of $P$  we take
$$ N_F\eqbd \{ \eta=(n_1, \ldots, n_k)\in (\Rr^k) : F\subseteq \{ x\in P : \eta x(\eqbd \sum_{j=1}^k n_j x_j )= \max_{y\in P}  \eta y  \}   \}  $$   
and $$ N(P)\eqbd \{ N_F : \;  F \;\hbox{\rm is a face of} \; P  \}.   $$

\noindent
Finally, consider the intersection of the negative of the normal fan $-N(P)$ and the nonnegative  quadrant $\Rr^k_+$. This splits the first quadrant $\Rr^k_+$ into several regions
$S_{\Lambda_j}$ according to which subsets $\Lambda_j$ of exponents $\lambda$ 
``dominate'' close to the considered component of the variety $V(p)$, in the
following sense:

\begin{definition}
Let $\Lambda_j$ be a subset of $\Lambda$ that determines a face of the Newton
polytope $P$ of $p$ such that the negative of its normal cone $-N(P)$ 
intersects $(\Rr^k)_+$ non-trivially
(not only at the origin). Define $S_{\Lambda_j} \in (\Rr^k)_+$ to be the 
set of all nonnegative row vectors $\eta$ such that 
\[ \eta{\lambda_1} = \eta{\lambda_2} < \eta {\lambda}, ~~\forall \lambda_1, \lambda_2 \in \Lambda_j, ~~\mbox{and}~ \lambda \in \Lambda \setminus \Lambda_j. \]
\end{definition}

Note that if $x_1$ through $x_k$ are small, then the exponential change of variables 
$x_j \mapsto -\log |x_j|$ gives rise to a correspondence between the nonnegative part 
of $-N(P)$ and the space of original variables $x_{[1{:}k]}$.  
We map back the sets $S_{\Lambda_j}$ into a neighborhood of $0$ in
 $\Rr^k$ by lifting:

\begin{definition}
Let $F_{\Lambda_j} \subseteq [-1,1]^k$ be the set of all points $x_{[1:k]}\in \Rr^k$ such that 
\[\eta\eqbd (-\log |x_1|, \ldots, - \log|x_k|) \in S_{\Lambda_j}. \]
\end{definition}
For any $j$, the closure of $F_{\Lambda_j}$ contains the origin in $\Rr^k$.
Given a point $x_{[1:k]} \in F_{\Lambda_j}$, and given $\eta=(n_1, n_2, \ldots, n_k) \in S_{\Lambda_j}$, for any $t\in (0,1)$, the vector $(x_1 t^{n_1}, \ldots, x_k t^{n_k})$ is in $F_{\Lambda_j} $.  Indeed, if $(-\log|x_1|, \ldots, -\log |x_k|)\in S_{\Lambda_j}$,
then so is  $(-\log|x_1|,\ldots,-\log|x_k|)-\log|t|\eta$, since all
equalities and inequalities that define $S_{\Lambda_j}$ will be preserved,
the latter because $\log|t|<0$. 

\begin{example} \label{example} Consider the following polynomial:
$$p(x_1,x_2,x_3)=x_2^8 x_3^{12} +x_1^2 x_2^2 x_3^{14}+ x_1^8 x_3^{12}+
x_1^6x_2^{14}+x_1^{10}x_2^6 x_3^4.$$ This polynomial is positive and easy to evaluate accurately; the reason we have chosen it is to illustrate the Newton polytope, its normal fan, and the sets $F_{\Lambda_j}$ and $S_{\Lambda_j}$ defined above.
 
For this example, 
\[
V(p) = \{x_1 = x_2 = 0\} \cup \{x_1 = x_3 = 0 \} \cup \{x_2 = x_3 = 0\}~.
\]
 
We examine the behavior of the polynomial near the $x_1=x_2=0$ component of the variety (i.e., we consider $x_3$ to be large). Note that only the first three monomial terms, $x_2^8 x_3^{12}$, $x_1^2 x_2^2 x_3^{14}$, and $x_1^8 x_3^{12}$ will play an important role, since if $x_1, x_2 \ll 1$, $x_1^6x_2^{14} \ll x_2^8  x_3^{12}$, respectively, $x_1^{10}x_2^6 x_3^4 \ll x_1^8 x_3^{12}$. 

We show below the Newton polytope $P$ of $p$ with respect to 
the variables $x_1$, $x_2$, its normal fan $N(P)$, the intersection
$-N(P)\cap R^2_+$, the regions $S_{\Lambda_j}$, and the regions $F_{\Lambda_j}$. 
\end{example}



\noindent
\begin{tabular}{llll}
\multicolumn{2}{l}{\epsfxsize=7.75cm \epsfbox{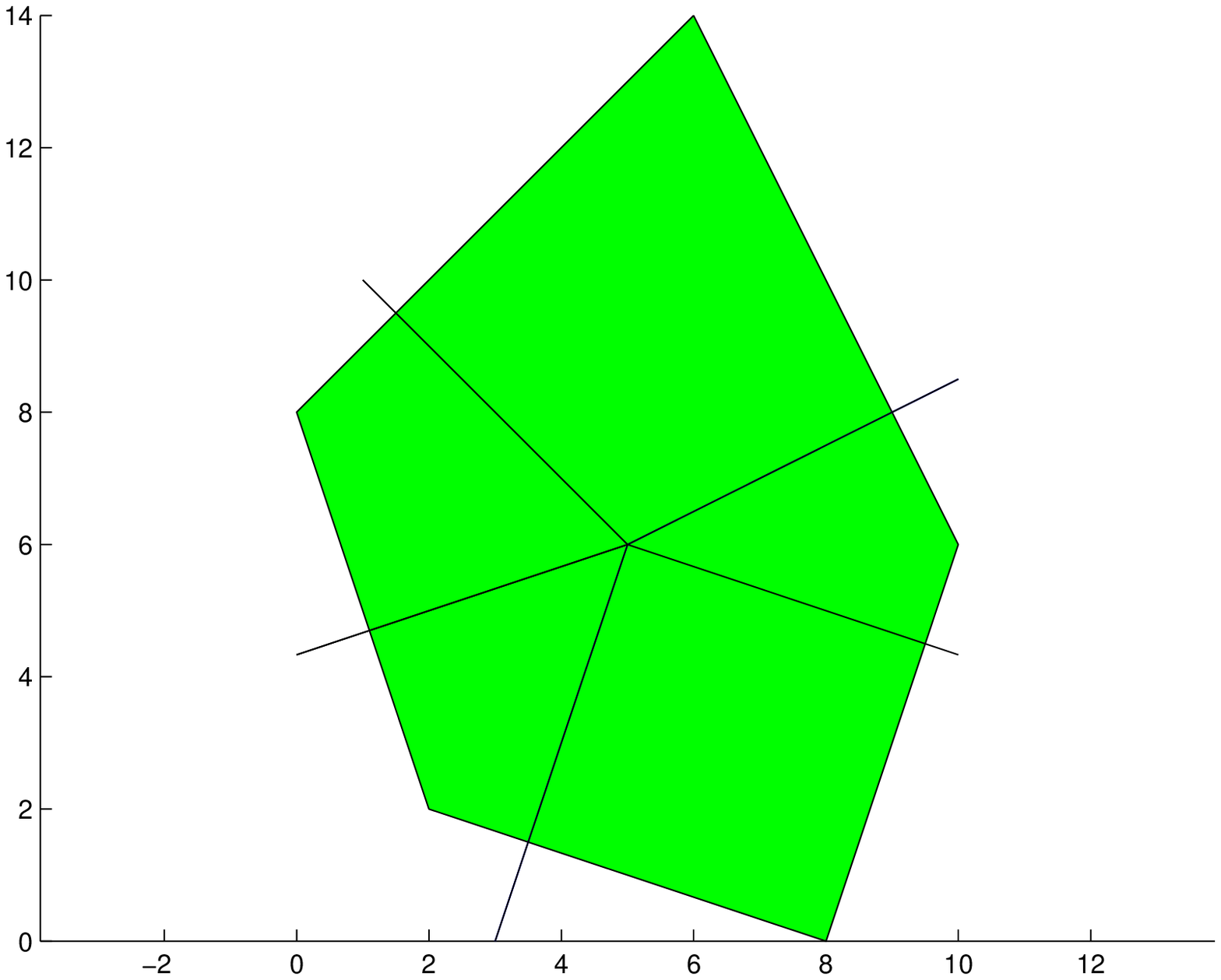}} 
&\multicolumn{2}{l}{\epsfxsize=7.75cm \epsfbox{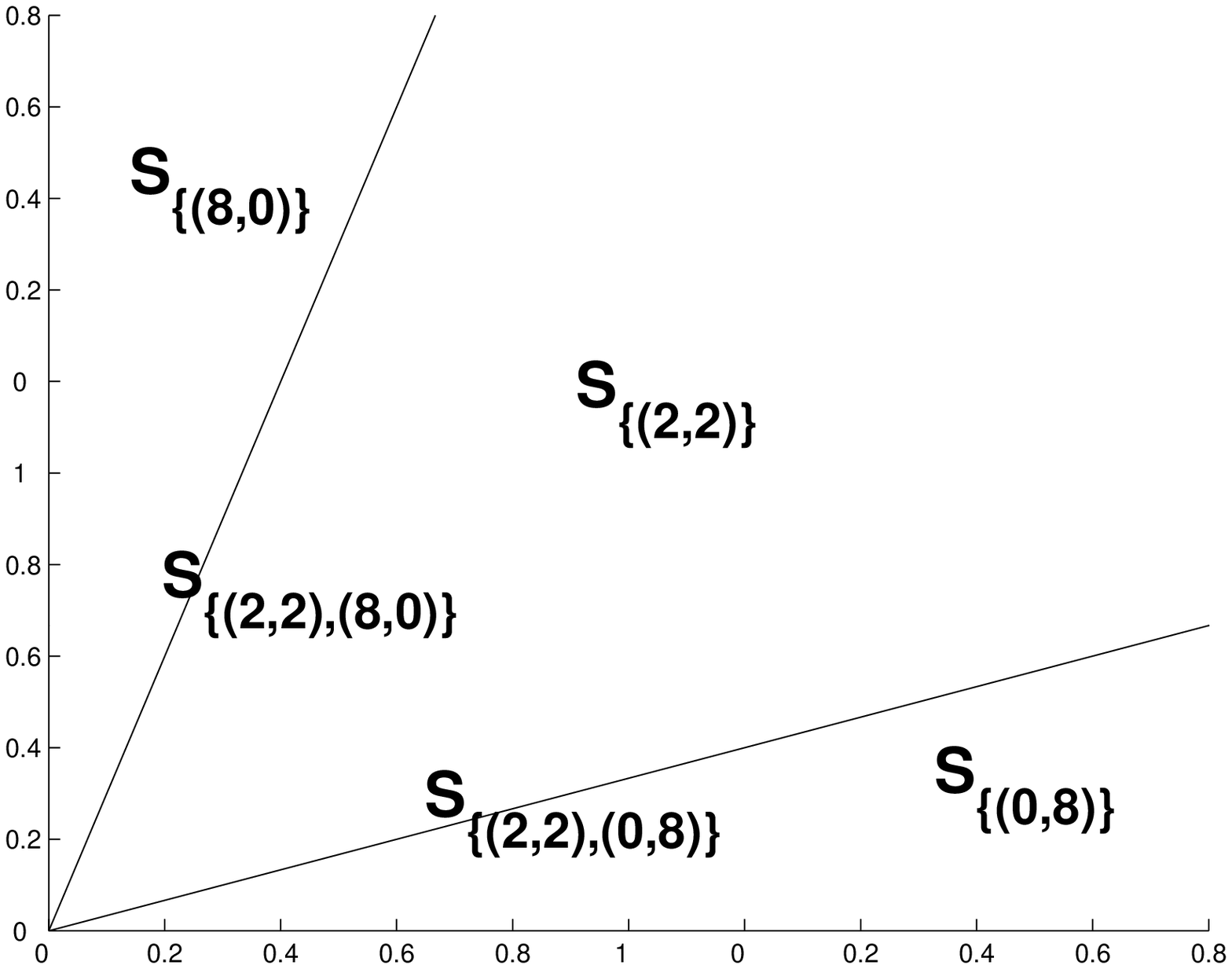}} \\
Figure 3: & The Newton polytope $P$ and its & 
Figure 4: & The intersection $-N(P)\cap \Rr^k_+$ and \\ 
\hspace{1.1cm} & normal fan $N(P)$ for Example \ref{example}.  
& & the regions $S_{\Lambda_j}$.  \\
\end{tabular} 

 


\noindent
\begin{tabular}{ll}
\multicolumn{2}{l}{\epsfxsize=11.0cm \epsfbox{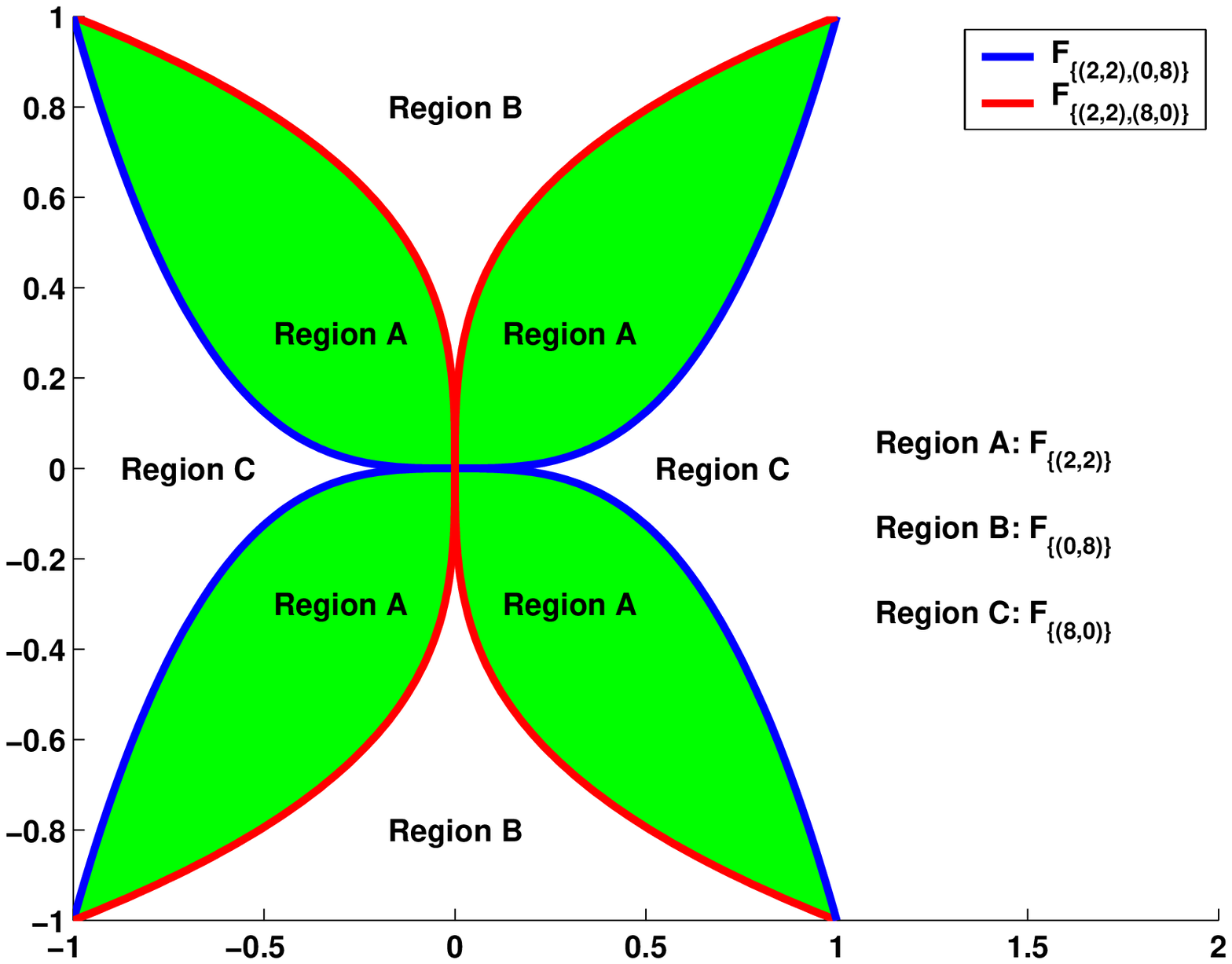}} \\
Figure 5: & The regions $F_{\Lambda_j}$.  \\ \\
\end{tabular}

\begin{definition}
%
We define the {\em dominant term\/} of $p(x)$ corresponding  to the component 
$x_1=\cdots=x_k=0$ and the region $F_{\Lambda_j}$ by
\[
p_{dom_j}(x) \eqbd \sum_{\lambda \in \Lambda_j} c_{\lambda} x_{[1:k]}^{\lambda} q_{\lambda}(x_{[k+1:n]})~.
\]  \end{definition}

The following observations about dominant terms are immediate. 

\begin{lemma} \label{leading} Let $\eta=(n_1, \ldots, n_k)\in S_{\Lambda_j}$ and let
$d_j\eqbd \sum_{\lambda_i\in \Lambda_j} \lambda_i n_i$.  Let $x^0$ be fixed and let 
$$ x(t)\eqbd (x_1(t), \ldots, x_n(t)), \qquad  x_j(t) \eqbd  \left\{ \begin{array}{ll} t^{n_j} x^0_j &  
j=1,\ldots, k,  \\ x^0_j, & j=k+1, \ldots, n. \end{array} \right. $$ 
Then $p_{dom_j}(x(t))$ has degree $d_j$ in $t$ and is the lowest degree 
term of $p(x(t))$ in $t$, that is
$$p(x(t))=p_{dom_j}(x(t))+o(t^{d_j})  \quad {\rm as}\;\; t\to 0, \qquad \deg_t p_{dom_j}(x(t))=d_j.  $$
\end{lemma}

\begin{corollary} Under the assumptions of Lemma~\ref{leading} suppose that
$p_{dom_j}(x^0)\neq 0$. Then 
$$\lim_{t\to 0} { p_{dom_j}(x(t)) \over  p(x(t))}=1.   $$  
\end{corollary}

The next question is whether the term $p_{dom_j}$ dominates indeed the remaining
terms of $p$ in the region $F_{\Lambda_j}$ in the sense that $p_{dom_j} (x)/p(x)$ is close to $1$ sufficiently close to $x_1 = \cdots = x_k = 0$. Indeed, we show that each dominant term $p_{dom_j}$ such that the convex hull of $\Lambda_j$ is a facet of the Newton polytope of $p$ and whose variety $V (p_{dom_j} )$
does not have a component strictly larger than the set $x_1 =\cdots = x_k = 0$ dominates the
remaining terms in $p$, not only in $F_{\Lambda_j}$, but in a certain \emph{slice} $\tilde{F}_{\Lambda_j}$ around $F_{\Lambda_j}$. These dominant
terms, corresponding to larger sets $\Lambda_j$ , are the useful ones, since they pick up terms relevant not
only in the region $F_{\Lambda_j}$ but also in its neighborhood. 

In Example~\ref{example} above, the useful dominant terms correspond to the regions $F_{\{(2,2),(8,0)\}}$ and $F_{\{(2,2),(0,8)\}}$ (the only relevant edges of the polygon). This points to the fact that we should be 
ultimately interested only in dominant terms corresponding to the facets, i.e., the
highest-dimensional faces, of the Newton polytope of $p$. Note that the convex hull
of $\Lambda_j$ is a facet of the Newton polytope $N$ if and only if the set
$S_{\Lambda_j}$ is a one-dimensional ray. 
 
The next lemma will be instrumental for our results in Section~\ref{sec_suff}.  It shows 
that each 
dominant term $p_{dom_j}$ such that the convex hull of $\Lambda_j$ is a facet of the 
Newton polytope of $p$  and whose variety $V(p_{dom_j})$ does not have a  component 
strictly larger than the  set $x_1=\cdots= x_k=0$ indeed dominates the remaining terms 
in $p$ in a certain ``slice'' $\widetilde{F}_{\Lambda_j}$  around $F_{\Lambda_j}$.

\begin{lemma} \label{true_dom} Let $p_{dom_j}$ be the dominant term of a homogeneous 
polynomial $p$ corresponding to the component $x_1=\cdots =x_k=0$ of the variety $V(p)$ 
and to the set $\Lambda_j$ whose convex hull is a facet of the Newton polytope $N$.
 
Let $\widetilde{S}_{\Lambda_j}$ be any closed pointed cone in $(\Rr^k)_+$ with vertex at
$0$ that does not intersect other one-dimensional rays $S_{\Lambda_l}$, $l\neq j$, and 
contains $S_{\Lambda_j} \setminus\{0\}$ in its interior. Let  $\widetilde{F}_{\Lambda_j}$
be the closure of the set
\begin{equation}
 \{x_{[1:k]} \in [-1,1]^k :  
 (-\log |x_1|, \ldots, - \log|x_k|) \in \widetilde{S}_{\Lambda_j} \}. \label{slices}
\end{equation} 
Suppose the variety $V(p_{dom_j})$ of $p_{dom_j}$ is allowable and intersects 
$\widetilde{F}_{\Lambda_j}$ 
only at $0$. Let $\| \cdot\|$ be any norm. Then, for any $\delta=\delta(j)>0$, there 
exists  $\varepsilon=\varepsilon(j)>0$  such that 
\begin{equation}
 \left|{p_{dom_j}(x_{[1:k]}, x_{[k+1:n]}) \over  p(x_{[1:k]}, x_{[k+1:n]})} -1  \right|<
\delta  \quad {\rm whenever} \;\; {\|x_{[1:k]}\|\over\|x_{[k+1:n]}\|}\leq \varepsilon \;\; {\rm and} \;\;
x_{[1:k]}\in \widetilde{F}_{\Lambda_j}.   \label{dom_inq}
\end{equation}
\end{lemma}

For a proof of Lemma \ref{true_dom}, the reader is referred to \cite{focm}. 

The above discussion of dominance was based on the transformation of a given irreducible
component of the variety to the form $x_1 = \cdots = x_k = 0$. We must reiterate that the identification
of dominant terms becomes possible only after a suitable change of variables $C$ is used to put a
given irreducible component into the standard form $x_1 = \cdots = x_k = 0$ and then the sets $\Lambda_j$ are
determined. Note however that the polynomial $p_{dom_j}$ is given in terms of the original variables,
i.e., as a sum of monomials in the original variables $x_q$ and sums/differences $x_q  \pm x_r$. We therefore
use the more precise notation $p_{dom_j ,C}$ in the rest of this section.

\begin{definition} Without loss of generality we can assume that any standard change of variables 
has the form
\begin{eqnarray*} \label{change*}
\!\!\!\!\!\!\!\!\!\!\! & & x = (x_{[1:k_1]}, x_{[k_1+1:k_2]}, \ldots, x_{[k_{l-1}+1:k_l]}) \mapsto \widetilde{x}=(\widetilde{x}_{[1:k_1]}, \widetilde{x}_{[k_1+1:k_2]}, \ldots, 
\widetilde{x}_{[k_{l-1}+1:k_l]}), \\
\!\!\!\!\!\!\!\!\!\!\!  & & \qquad {\rm where} ~~\widetilde{x}_{k_m+1}\eqbd x_{k_m+1}, \;\; \widetilde{x}_{k_m+2}\eqbd 
x_{k_m+2} -\sigma_{k_m+2} x_{k_m+1}, \;\; \ldots, \;\; \\
\!\!\!\!\!\!\!\!\!\!\! & & \widetilde{x}_{k_{m+1}} \eqbd  x_{k_{m+1}} -\sigma_{k_{m+1}} x_{k_{m+1}},~~
k_0\eqbd 0, \quad \sigma_r=\pm 1 \;\;\; \hbox{\rm for all pertinent}\;\; r ~.
\end{eqnarray*}
\end{definition}

\vspace{.2cm}

Note also that we can think of the vectors $\eta\in S_{\Lambda_j}$ as being indexed by 
integers $1$ through $k_l$, i.e., $\eta=(n_1, \ldots, n_{k_l})$. Moreover, to define 
pruning in the next subsection we will assume that
\begin{equation}
n_{k_m+1}\leq n_r \quad \hbox{\rm for all}\;\; \;r=k_m+2, \ldots, k_{m+1}
\quad \hbox{\rm and for all} \;\; m=0, \ldots, l-1.
\label{exp_cond}
\end{equation}

\subsubsection{Pruning}  \label{sec_prune}

We show here how to convert an accurate algorithm that evaluates a polynomial $p$
into an accurate  algorithm that evaluates a selected dominant term $p_{dom_j,C}$. This will imply that being able to evaluate dominant terms accurately is a necessary condition for being able to evaluate the original polynomial accurately.

This process, which we will refer to as {\em pruning,\/}  will consist of deleting some 
vertices and edges and redirecting certain other edges in the DAG that represents the 
algorithm. 
We explain the pruning process informally and through an example; for the rigorous definition, see \cite{focm}.

Starting at the sources, we process each node provided that both of its inputs have been processed already (acyclicity insures that this can be done). Then, at any node $u$ which performs an addition or subtraction of two inputs from nodes $v$ and $w$ of different degrees, we delete the node and the in-edge from the input of smaller degree (say $v$) and redirect the out-edge from $u$ to $w$ (the node with the larger degree output). Then we go backward and delete all nodes and/or edges on that sub-DAG, up to the source nodes.
We denote the output of the pruned algorithm by  $p_{dom_j,C, comp}(x,\delta)$.
\vskip 0.1cm

We illustrate this process below.

\vspace{.5cm}

\begin{tabular}{lc}
\multicolumn{2}{c} {\epsfxsize=8.0cm \epsfbox{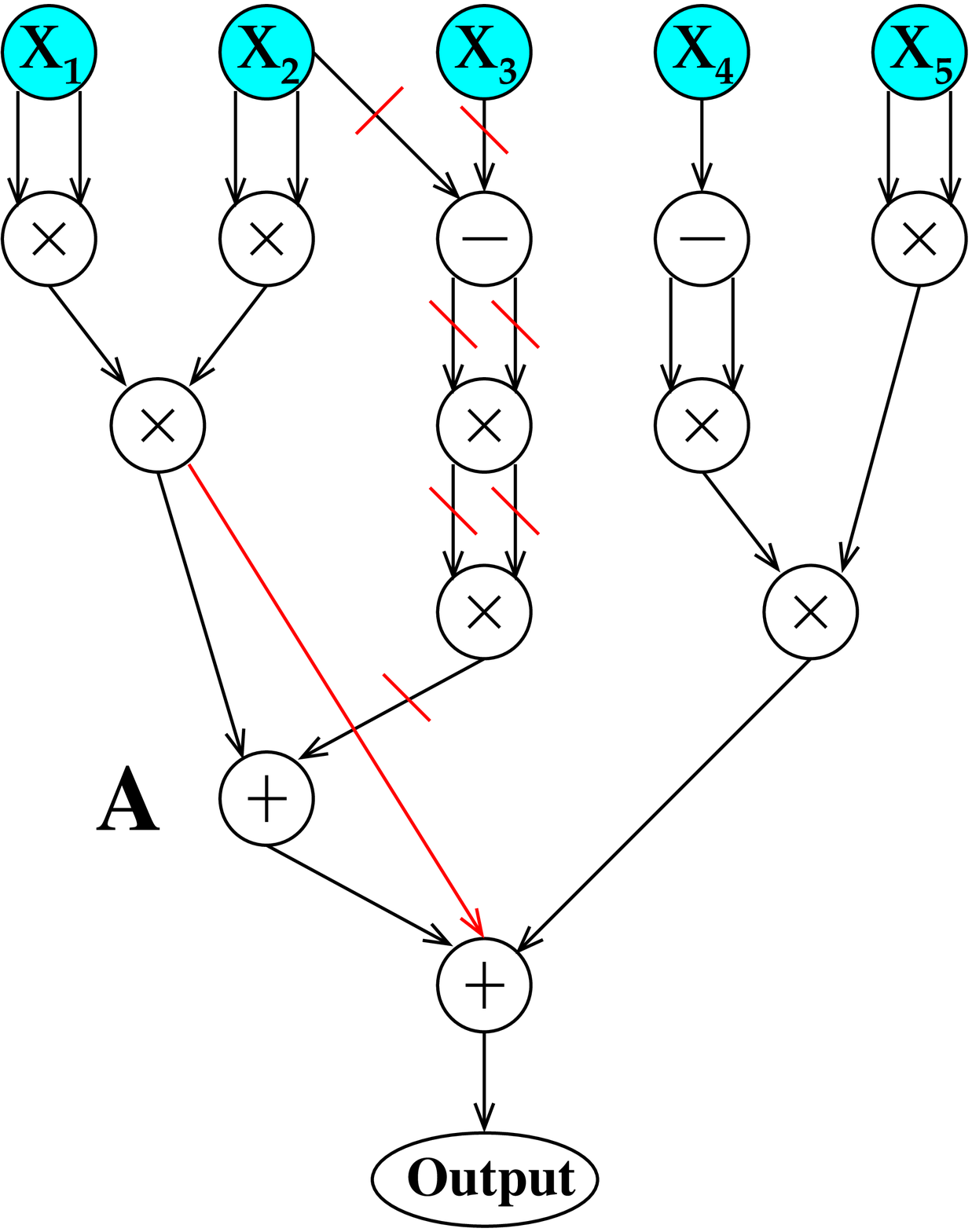}} \\ \\
\hspace{-1cm} Figure 6: & Pruning an algorithm for $p(x)= x_1^2 x_2^2 + (x_2-x_3)^4+(x_3-x_4)^2x_5^2 $.  \\ \\
\end{tabular}

\begin{example} Figure~6 shows an example of pruning an algorithm that evaluates the 
polynomial  $$   x_1^2 x_2^2 + (x_2-x_3)^4+(x_3-x_4)^2x_5^2  $$
using the substitution $$  (tx_1,x_2,tx_3+x_2, tx_4+x_2,x_5) $$
near the component $$ x_1=0, \;\; x_2=x_3=x_4.$$

\noindent The result of pruning is an algorithm that
evaluates the dominant term $$ x_1^2x_2^2+(x_3-x_4)^2x_5^2.$$ 
The node $A$ has two sub-DAGs leading to it; the right one (going back to the sources $x_2$ and $x_3$) is pruned due to the fact that it computes $(x_2 - x_3)^4$, a quantity of order $O(t^4)$, whereas the other produces $x_1^2 x_2^2$, a quantity of order $O(t^2)$. 

The output of the original algorithm is given by  
\begin{eqnarray*}
p_{comp}(x, \delta) &=& 
 \left [  \left(   x_1^2(1+\delta_1)x_2^2(1+\delta_2) (1+\delta_3) \right. \right.
\\
& & \left. + (x_2-x_3)^4(1+\delta_4)^4(1+\delta_5)^2(1+\delta_6)\right ](1+\delta_7)  \\
&&  + \left [(x_3-x_4)^2(1+\delta_8)^2(1+\delta_9)x_5^2(1+\delta_{10})(1+\delta_{11})\right ] (1+\delta_{12}).
\end{eqnarray*}
The output of the pruned algorithm is
\begin{eqnarray*}
p_{dom_j, C, comp}(x, \delta) &=&  \left [  x_1^2x_2^2 (1+\delta_1)(1+\delta_2) (1+\delta_3))(1+\delta_7)  + (x_3-x_4)^2 x_5^2 \right.\\
& & 
\left. \times (1+\delta_8)^2(1+\delta_9)(1+\delta_{10})(1+\delta_{11})\right]
(1+\delta_{12}).
\end{eqnarray*} \end{example}

We formalize the main result regarding the pruning process below. 

\begin{theorem} \label{p=>pdom}  Suppose a non-branching algorithm evaluates
a polynomial $p$ accurately on $\Rr^n$ by computing $p_{comp}(x,\delta)$. Suppose
$C$ is a standard change of variables~(as in Definition \ref{change*}) associated with an
irreducible component of $V(p)$. Let $p_{dom_j,C}$ be one of the corresponding
dominant terms of $p$ and let $S_{\Lambda_j}$ satisfy~(\ref{exp_cond}). Then the 
pruned algorithm 
with output  $p_{dom_j,C, comp}(x,\delta)$ evaluates $p_{dom_j,C}$ accurately on $\Rr^n$. 
In other  words, being able to compute all such $p_{dom_j,C}$ for all components of the 
variety $V(p)$  and all standard changes of variables $C$ accurately is a  necessary condition for computing $p$ accurately.
\end{theorem}

\subsubsection{Sufficiency of evaluating dominant terms}  \label{sec_suff}

Our next goal is to prove a converse to Theorem~\ref{p=>pdom}; however, strictly
speaking, the results that follow do not provide a true converse, since branching
is needed to construct an algorithm that evaluates a polynomial $p$ accurately
from algorithms that evaluate its dominant terms accurately. Recall that Theorem~\ref{p=>pdom} involves non-branching algorithms.

We make two assumptions: that our polynomial $p$ is homogeneous and irreducible. The latter assumption effectively reduces the problem to that of accurate evaluation of a nonnegative polynomial, due to the following lemma.

\begin{lemma} If a polynomial $p$ is irreducible and has an allowable variety $V(p)$,
then it is either a constant multiple of a linear form that defines an allowable
hyperplane, or it does not change its sign in $\Rr^n$.
\end{lemma}

Hence, we can restrict ourselves to the case of a homogeneous, irreducible, non-negative polynomial over the entire $\mathbb{R}^n$. For this case, we have the following theorem.

\begin{theorem} \label{pdom=>p} Let $p$ be a homogeneous nonnegative polynomial 
whose variety $V(p)$ is allowable. Suppose that all dominant terms $p_{dom_j,C}$ 
for all components of the variety $V(p)$, all standard changes of variables $C$ 
and all subsets $\Lambda_j$ satisfying~(\ref{exp_cond}) are accurately evaluable. 
Then there exists a branching  algorithm that evaluates $p$ accurately over $\Rr^n$.  
\end{theorem}

\vspace{.25cm}

\begin{proof}{Theorem \ref{pdom=>p}} We first show how to evaluate $p$ accurately in a neighborhood of each irreducible
component of its variety $V (p)$. We next evaluate $p$ accurately off these neighborhoods of $V (p)$. The
final algorithm will involve branching depending on which region the input belongs to, and the
subsequent execution of the corresponding subroutine.

Consider a particular irreducible component $V_0$ of the variety $V (p)$; using a standard change
of variables $C$, we map $V_0$ to a set of the form $\tilde{x}_1 = \cdots = \tilde{x}_k = 0$. We create an $\epsilon$-neighborhood
of $V_0$ where we can evaluate $p$ accurately; this neighborhood is built up from semi-algebraic $\epsilon$-neighborhoods. More precisely, for each $V_0$, we can find a collection $(S_j)$ of semi-algebraic sets, all determined by
polynomial inequalities with integer coefficients, and the corresponding numbers $\epsilon_j$, so that the polynomial $p$ can be evaluated with desired accuracy $\eta$ in each $\epsilon_j$ -neighborhood of $V_0$ within the
piece $S_j$. Moreover, testing whether a particular point $x$ is within $\epsilon_j$ of $V_0$ within $S_j$ can be done
by branching based on polynomial inequalities with integer coefficients.

The final algorithm will be organized as follows. Given an input $x$, determine by branching
whether $x$ is in $S_j$ and within the corresponding $\epsilon_j$ of a component $V_0$. If that is the case, evaluate $p(x)$ using the algorithm that is accurate in $S_j$ in that neighborhood of $V_0$. For $x$ not in any of
the neighborhoods, evaluate $p$ by Horner's rule. Since the polynomial $p$ is strictly positive off the neighborhoods of the components of its variety, the reasoning of Section \ref{sec_PositivePolys} applies, showing that the
Horner's rule algorithm is accurate. If $x$ is on the boundary of a set $S_j$, any applicable algorithm
will do, since the inequalities we use are not strict. Thus the resulting algorithm for evaluating $p$
will have the desired accuracy $\eta$.
 \end{proof}

\subsubsection{Obstacles to a complete inductive procedure} \label{obst}

The results of the previous sections suggest the existence of an inductive procedure that could be used to determine whether or not a given polynomial is accurately evaluable 
by reducing the problem for the original polynomial $p$ to the same problem for
its dominant terms, then their dominant terms, and so forth, going all the way to ``base'' cases: monomials or other polynomials that are easy to analyze. In order to work, the dominant terms would have to be simpler, or smaller, by some measure, than the original polynomial; this would require finding an induction variable that gets reduced at each step.

The most obvious two choices are the number of variables or the degree of the
polynomial under consideration; unfortunately, there are cases when both fail to decrease. Furthermore, the dominant term may even coincide with the
polynomial itself. For example, if $$p(x)=A(x_{[3:n]}) x_1^2 +B(x_{[3:n]})
 x_1 x_2 +C(x_{[3:n]})x_2^2$$ where $A$, $B$, $C$ are nonnegative polynomials 
in  $x_3$ through $x_n$, then the only useful dominant term of $p$ in the neighborhood
of the set $x_1=x_2=0$ is the polynomial $p$ itself. For this case, analyzing the dominant term yields no progress whatsoever.

Another possibility is induction on domains or slices of space, but we do not yet envision 
how to make this idea precise, since we do not know exactly when a given 
polynomial is accurately evaluable on a given domain. 

Further work to establish a full decision procedure is therefore highly desirable. 

\subsection{Extended arithmetic} \label{sec_extended}

In this section, we consider adding ``black-box'' real or complex polynomial operations to the basic, traditional model. We describe this type of operations below.

\begin{definition} We call a black-box operation any type of operation that takes a number of inputs (real or complex) $x_1, \ldots, x_k$ and produces an output $q$ such that $q$ is a polynomial in $x_1, \ldots, x_k$. \end{definition}

\begin{example} $q(x_1, x_2, x_3) = x_1 + x_2 x_3$.
\end{example}

Note that $+, -$, and $\cdot$ are all black-box operations on two inputs.

Consider a fixed set of multivariate polynomials $\{ q_j : j\in J\}$ with real or complex inputs (perhaps infinite). In the extended arithmetic model, the operations allowed are the black-box operations $q_1, \ldots, q_k$, and negation. With the exception of negation, which is exact, all the others yield $rnd(op(a_1, \ldots, a_l)) = op(a_1, \ldots, a_l)(1+\delta)$, with $|\delta|< \epsilon$ ($\epsilon$ here is the machine precision). We consider the same arithmetical models as in Section~\ref{sec_models}, with this extended class of operations.

\subsubsection{Necessity: real and complex}

In order to analyze the way in which the necessity condition for having an allowable variety (Theorem \ref{conj}) changes under these extended assumptions, we need to introduce a new, more general definition of allowability.

Essentially, a black box for computing $p$ can be used for computing other polynomials, namely all the polynomials obtainable from $p$ via permuting, repeating, negating, and zeroing some subset of the variables. Therefore each black box accounts for a potentially larger set of polynomials that can be evaluated with {\sl a single\/} rounding error, using that black box, and we must consider all of them in our analysis. Note that in the traditional case (when we had addition, subtraction, and multiplication of two numbers as our black boxes) our set of three operations was closed under the aforementioned changes. 

The definition below formalizes the set of polynomials obtainable from a given one, through this process of negation, repetition, permutation, and zeroing of variables.

Recall that we denote by $\S$ the space of variables (which may be either $\Rr^n$ or $\C^n$). From now on we will denote the set $\{1, \ldots, n\}$ by $\k$, and the set of pairs $(i,j) \in \k \times \k$ such that $i<j$ by $\k^2_{<}$.

\begin{definition} \label{all_subv}
Let $p(x_1, \ldots, x_n)$ be a multivariate polynomial over $\S$ with variety $V(p)$. Let $\k_Z \subseteq \k$, and let $\k_D, \k_S \subseteq \k^2_{<}$ . Modify $p$ as follows: impose conditions of the type $Z_i$ for each $i \in \k_Z$, and of type $D_{ij}$, respectively $S_{ij}$, on all pairs of variables in $\k_D$, respectively $\k_S$. Rewrite $p$ subject to those conditions (e.g., set $X_i = 0$ for all $i \in \k_Z$), and denote it by $\tilde{p}$, and denote by $\k_R$ the set of remaining independent variables (use the convention which eliminates the second variable in each pair in $\k_D$ or $\k_S$).

Choose a set $T \subseteq \k_R$, and let
$$ V_{T, \k_Z, \k_D, \k_S} (p) = \cap_{\alpha} V(q_{\alpha})~, $$
where the polynomials $q_{\alpha}$ are the coefficients of the expansion of $\tilde{p}$ in the variables $x_T$:
$$ \tilde{p} (x_1, \ldots, x_k) = \sum_{\alpha} q_{\alpha} x_{T}^{\alpha}~, $$
with $q_{\alpha}$ being polynomials in $x_{\k_R \setminus T}$ only.

Finally, let $\k_N$ be a subset of $\k_R \setminus T$. We negate each variable in $\k_N$, and let $V_{T, \k_Z, \k_D, \k_S, \k_N}(p)$ be the variety obtained from $V_{T, \k_Z, \k_D, \k_S}(p)$, with each variable in $\k_N$ negated. 
\end{definition}

\vspace{.25cm}

For simplicity, we denote a set $(T, \k_Z, \k_D, \k_S, \k_N)$ by 
$\ii$.

We illustrate this process by the following example. 

\begin{example} Let $p(x,y,z) = x+y \cdot z$ (the fused multiply-add). We
record below some of the possibilities for 
the subvarieties $V_{\ii}(p)$; the sets $\ii~=~(T,\k_Z,\k_D,\k_S,\k_N)$ are implicit.


\begin{enumerate}
\item[{\mathversion{bold} $\diamond$}] $V(p(x, 0, z)) ~~=  \{x=0\}$,
\item[{\mathversion{bold} $\diamond$}] $V(p(x,x,x)) ~= \{x= 0\} \cup \{x = -1\}$,
\item[{\mathversion{bold} $\diamond$}] $V(p(0, y, z)) ~~= \{y = 0\} \cup \{z=0\}$,
\item[{\mathversion{bold} $\diamond$}] $V(p(x, y, -x)) = \{x = 0\} \cup \{y=1\}$,
\item[{\mathversion{bold} $\diamond$}] $V(p(x, y, y)) ~~= \{x +y^2 = 0\}$,
\item[{\mathversion{bold} $\diamond$}] $V(p(x, y, -z)) = \{x -yz = 0\}$, etc.
\end{enumerate} \end{example}

We include the ``traditional'' operations in the arithmetic by defining $q_{-2}(x_1, x_2) = x_1 x_2$, $q_{-1}(x_1, x_2) = x_1+x_2$, and $q_0(x_1, x_2) = x_1 - x_2$, and note that the sets 
\begin{eqnarray} \label{q-unu}
& 1.&  Z_i =\{x~:~ x_i ~=~ 0\}~, \\
\label{q-doi}
& 2.&  S_{ij} = \{x~:~ x_i+x_j ~=~ 0\}~, \\
\label{q-trei}
& 3.&  D_{ij} = \{x~:~ x_i - x_j ~=~ 0\} 
\end{eqnarray}
 describe all non-trivial sets of type $V_{\ii}$, for $q_{-2}, ~q_{-1}$, and $q_0$.

\vspace{.25cm}

We will assume from now on that the black-box operations $q_j$ with $j \in J$ ($J$ may be infinite, and $\{-2, -1, 0\} \subset J$) are given and fixed.

\vspace{.25cm}

\begin{definition} \label{q_allow} 
We call any set $V_{\ii}(q_j)$ with $\ii= (T, \k_Z, \k_D, \k_S, \k_N)$ as defined above and $q_j$ a black-box operation
{\sl basic $q$-allowable}.

We call any set $R$ {\sl irreducible $q$-allowable} if it is an irreducible component of a (finite) intersection of basic $q$-allowable sets, i.e., when $R$ is irreducible and 
\[ 
R \subseteq \cap_{l} ~Q_l~,
\]
where each $Q_l$ is a basic $q$-allowable set.

We call any set $Q$ {\sl $q$-allowable} if it is a (finite) union of irreducible $q$-allowable sets, i.e.,\[
Q = \cup_{j} R_j~,
\]
where each $R_j$ is an irreducible $q$-allowable set. 

Any set $R$ which is not $q$-allowable we call $q$-unallowable. 
\end{definition}

Note that the above definition of $q$-allowability is closed under taking union, intersection, and irreducible components. This parallels the definition of allowability for the classical arithmetic case -- in the classical case, every allowable set was already irreducible (being an intersection of hyperplanes).

\begin{definition}
Given a polynomial $p$ with $q$-unallowable variety $V(p)$, consider 
all sets $W$ that are $q$-allowable (as in Definition~{\sl\ref{q_allow}}), and subtract from $V(p)$ those $W$ for which $W \subset V(p)$. 
We call the remaining subset of the variety {\em points in general position\/} and denote it by $\G(p)$.
\end{definition}

Since $V(p)$ is $q$-unallowable, $\G(p)$ is non-empty.

\begin{definition}
Given $x\in \S$, define the set $\qallow(x)$ as the intersection of
all basic $q$-allowable sets going through $x$:
$$ \qallow(x) \eqbd \cap_{j \in J}  \left( \cap_{\ii~:~ x\in V_{\ii}(q_j)} ~~V_{\ii}(q_j) \right) , $$
for all possible choices of $\ii$. The intersection in parentheses is $\S$  whenever $x \notin V_{\ii}(q_j)$ for all $\ii$.
\end{definition}

Note that when $x \in \G(p)$, $\qallow(x) \not\subseteq \G(p)$.

We can now state our necessity condition.

\begin{theorem} \label{gen_result} Given the black-box operations $\{q_j: j \in J\}$, and the model of arithmetic described above, let $p$ be a polynomial defined over a domain $\mathcal{D} \subset \S$. Let $\mathcal{G}(p)$ be the set of points in general position on the variety $V(p)$. 
If there exists $x \in \mathcal{D} \cap \mathcal{G}(p)$ such that $\qallow(x) \cap \Int(\mathcal{D}) \neq \emptyset$, then $p$ is not accurately evaluable on $\mathcal{D}$. \end{theorem}

\begin{proof}{Theorem \ref{gen_result}} The proof mimics the proof of Theorem \ref{conj}; once again, we trace back zeros to what we now call $q$-allowable conditions, and make use of the DAG structure of the algorithm. In the non-branching case, we obtain that  if the algorithm is run on an input $x \in G(p)$, then either $p_{comp}(x, \delta) \neq 0$ for almost all $\delta$, or $p_{comp}(y, \delta) =0$ for all $y \in \allow(x) \setminus V(p)$ and for all $\delta$.
The proof for the branching case is again a refinement of the proof for the non-branching one. \end{proof}

Note that, if we consider only  algorithms without branching, Theorem \ref{gen_result} remains true in the tighter case when we drop the irreducibility constraint from the definition of allowability.

We can also show that, arbitrarily close to any point $x \in \G(p)$, we can find sets $S$ of positive measure such that the relative accuracy of the algorithm when run with inputs in $S$ is either $1$ or $\infty$; a result identical to Corollary \ref{in_sfirsit} can also be proved for the extended arithmetic case. 

\subsubsection{Sufficiency: the complex case} \label{suf_c_bb}

In this section we obtain a sufficiency condition for the accurate evaluability of a complex polynomial, given a black-box arithmetic with operations $\{q_j ~|~ j \in J\}$ ($J$ may be an infinite set). 


Throughout this section, we assume our black-box operations include $q^{c}$, which consists of multiplication by a complex constant: $q^c(x) = c\cdot x$. Note that this operation is natural, and can be performed accurately given only a suitably accurate approximation of $c$.


We believe that the sufficiency condition we obtain here is not a necessary one, in general--but it does subsume the sufficiency condition we found for the basic complex case with classical arithmetic $\{+, -, \cdot\}$. 

\begin{theorem}[General case]\footnote{This condition was stated in a slightly weaker form in \cite{focm}.} \label{q-suff-c1} Given a polynomial $p~:~\C^n~\rightarrow~\C$ with $V(p)$ a finite union of irreducible varieties $V_{\mathcal{I}}(q_j)$, for $j \in J$, and $\mathcal{I}$ as above, then $p$ is accurately evaluable. \end{theorem} 

\begin{theorem}[Affine case] \label{q-suff-c2} If all black-box operations $q_j$, $j \in J$ are affine, then a polynomial $p~:~\C^n~\rightarrow~\C$ is accurately evaluable iff $V(p)$ is a union of varieties $V_{\ii}(q_j)$, for $j \in J$ and $\ii$ as in Definition~\ref{all_subv}.
\end{theorem}

The proofs follow easily from Lemma \ref{q-factors}.

\begin{lemma} \label{q-factors}If all varieties $V_{\mathcal{I}}(q_j))$ in the union defined by $V(p)$ are irreducible (in particular, if they are affine), then $p$ is a product 
$p=c \prod_j p_j$, where each $p_j$ is a power of $q_j$ or a polynomial obtained from $q_j$ by repeating, negating, or zeroing some of the variables; $c$ is a complex constant. The argument is identical to the one we gave for the proof of Corollary \ref{factors}, and it hinges on the irreducibility of the varieties $V_{\mathcal{I}}(q_j))$ in the union.
\end{lemma}

Note that Theorem~\ref{q-suff-c2} is a more general necessary and sufficient condition than Theorem~\ref{sufficiency_c}, which only considered having $q_{-2}, q_{-1}$, and $q_0$ as operations, and restricted the polynomials to have integer coefficients (thus eliminating the need for $q^c$). 

\subsection{Numerical linear algebra consequences} \label{sec_conseq}

Here we examine the results of Section \ref{sec_Plamen}, in light of 
Section \ref{sec_OlgaIoana}. We take another look at Table~\ref{table1}, 
explaining the strong ``No'' entries there. Those entries mean 
that no accurate algorithms exist even given an arbitrary
set of black-box operations of bounded degree or with a bounded number of arguments. In other 
words, arbitrary precision arithmetic is needed for their accurate
solution.  This is the case for Toeplitz matrices because, as discussed 
earlier,  we cannot  evaluate their determinants accurately, and 
determinants are necessary to get the indicated entries accurately. 
Fully off-diagonal submatrices of diagonally dominant matrices
are completely unstructured matrices, and so with irreducible
determinants of unbounded degree. 
The same is true of M-matrices, except that the submatrix
entries are nonpositive.
Minors of submatrices of non-TN Vandermonde have factors that
are general Schur functions of arbitrary arguments, which can be
irreducible of unbounded degree. We suspect that many other
entries should also be ``No''.

\ignore{
As discussed in Section \ref{sec_Plamen}, being able to compute the determinant (respectively selected minors) is a necessary (respectively sufficient) condition for many of the linear algebra problems discussed there, such as the LDU and SVD decompositions.
}

\subsubsection{Validation of our results}


If we examine the matrix classes in Table~\ref{table1}, we see that their determinants are rational functions whose sets of zeros and of poles are allowable in traditional arithmetic. By considering numerators and denominators of these rational functions separately we see that both can be computed accurately (and then, provided that the denominator is not $0$, their ratio can be computed accurately). Incorporating division more formally into our model to identify necessary and sufficient conditions for accurate evaluability of rational functions is the subject of ongoing work.

\subsubsection{Negative results: accurate evaluation is impossible}


Here we examine two classes of matrices for which some or all linear algebra operations are impossible given
 any set of black boxes with a bounded number of arguments: Toeplitz and various classes of Vandermonde that
 we define later.

We prove our results by reducing the problem of doing accurate linear algebra to that of accurately evaluating the determinant and certain minors (recall that the latter is a necessary condition for the former). What these results say roughly that, if one wants to construct an accurate algorithm for finding the inverse that works for Toeplitz or Vandermonde matrices as a class, one needs to use arbitrary precision (more on this in Section \ref{sec_OtherModels}).

We start by examining a more general problem. 
If the determinants $p_n(x) = \det M^{n \times n}(x)$
of a class of $n$-by-$n$ structured matrices $M$ do not satisfy the 
necessity conditions described in Theorem~\ref{gen_result} for {\em any\/} 
enumerable set of black-box operations (perhaps with other properties, like bounded degree), then we can conclude that 
accurate algorithms of the sort described in the above citations
are impossible.

In particular, to satisfy these necessity conditions would require that
the varieties $V(p_n)$ be allowable (or $q$-allowable).
For example, if $V$ is a Vandermonde matrix, then 
$\det(V) = \prod_{i<j} (x_i-x_j)$ satisfies this condition, using
only subtraction and multiplication.

The following theorem states a negative condition 
(which guarantees impossibility of existence for algorithm using {\em any\/} enumerable set of
black-box operations of bounded degree).

\begin{theorem}
\label{Thm_StructuredMatrixImpossibility}
Let $M(x)$ be an $n$-by-$n$ structured complex matrix with determinant
$p_n(x)$ as described above. Suppose that for any $n$, $p_n(x)$ has an irreducible
factor $\hat{p}_n(x)$ whose degree goes to infinity as $n$ goes
to infinity. Then for any enumerable set of black-box arithmetic
operations of bounded degree, for sufficiently large $n$ it is impossible to 
accurately evaluate $p_n(x)$ over the complex numbers.
\end{theorem}

\begin{proof}
Let $q_1,...,q_m$ be any finite set of black-box operations.
To obtain a contradiction, suppose the complex variety
$V(p_n)$ satisfies the necessary conditions of Theorem~\ref{gen_result},
i.e., that $V(p_n)$ is allowable. 
This means that $V(p_n)$, which
includes the hypersurface $V(\hat{p}_n)$ as an irreducible component, 
can be written as the union of irreducible 
$q$-allowable sets (by Definition~\ref{q_allow}).
This means that $V(\hat{p}_n)$ must
itself be equal to an irreducible 
$q$-allowable set (a hypersurface), since representations
as unions of irreducible sets are unique.
The irreducible $q$-allowable sets of codimension 1 are
defined by single irreducible polynomials, which are
in turn derived by the process of setting variables
equal to one another, to one another's negation, or zero
(as described in Definitions~\ref{all_subv} and~\ref{q_allow}),
and so have bounded degree.
This contradicts the unboundedness of the degree of $V(\hat{p}_n)$.
\end{proof}

In the next theorems we apply this result to the set of 
Toeplitz matrices. We use the following notation.
Let $T$ be an $n$-by-$n$ Toeplitz matrix, with
$x_j$ on the $j$-th diagonal, so $x_0$ is on the main diagonal,
$x_{n-1}$ is in the top right corner, and $x_{1-n}$ is in the bottom left
corner. We give the following result without proof; for a proof, see \cite{focm}.

\begin{theorem}
\label{Thm_Toeplitz_C}
The determinant of a Toeplitz matrix $T$ is irreducible over any field.
\end{theorem}

Therefore, for complex Toeplitz matrices, we have the following corollary.

\begin{corollary} \label{cor2}
The determinants of the set of complex Toeplitz matrices 
cannot be evaluated accurately using any enumerable set of bounded-degree 
black-box operations.
\end{corollary}

In the real case, irreducibility of $p_n$ 
is not enough to conclude that $p_n$ cannot be
evaluated accurately, because $V_{\Rr}(p_n)$ may
still be allowable (and even vanish). 
So we consider another necessary condition for
allowability: Since all black boxes have a finite
number of arguments, their associated codimension-1
irreducible components
must have the property that
whether $x \in V_{\ii}(q_j)$ depends on only 
a finite number of components of $x$.
Thus to prove that the hypersurface $V_{\Rr}(p_n)$ is not allowable,
it suffices to find at least one regular point $x^*$ in $V_{\Rr}(p_n)$
such that the tangent hyperplane at $x^*$ 
is not parallel to
sufficiently many coordinate directions, i.e., membership in
$V_{\Rr}(p_n)$ depends on more variables than any $V_{\ii}(q_j)$.
This is easy to do for real Toeplitz matrices.

\begin{theorem}
\label{thm_Toeplitz_R}
Let $V$ be the variety of the determinant of real singular Toeplitz matrices.
Then $V$ has codimension 1, and at almost all regular points,
its tangent hyperplane is parallel to no coordinate directions.
\end{theorem}

\begin{corollary} \label{cor2R}
The determinants of the set of real Toeplitz matrices 
cannot be evaluated accurately using any enumerable set of bounded-degree 
black-box operations.
\end{corollary}

Proofs of these results can be found in \cite{focm}. Corollaries  \ref{cor2} and \ref{cor2R} imply that 
accurate linear algebra (in the sense of Section \ref{sec_Plamen}) is impossible on the class of Toeplitz 
matrices (either real or complex) in bounded precision.

We consider now the class of polynomial Vandermonde matrices $V$,
where $V_{ij} = P_{j-1}(x_i)$ is a polynomial function of $x_i$,
with $1 \leq i,j \leq n$.
This class includes the standard Vandermonde 
(where $P_{j-1}(x_i) = x_i^{j-1}$) and many others.

Consider first  a generalized Vandermonde matrix where
$P_{j-1}(x_i) = x_i^{j-1+ \lambda_{n-i}}$
with $0 \leq \lambda_1 \leq \lambda_2 \leq \cdots \leq \lambda_n$.
The tuple $\lambda = (\lambda_1 , \lambda_2 , ... , \lambda_n)$ is called a {\em partition.\/}
Any square submatrix of such a generalized Vandermonde matrix 
is also a generalized Vandermonde matrix.
A generalized Vandermonde matrix is known to have determinant of the form
$s_{\lambda}(x) \prod_{i<j} (x_i-x_j)$ where $s_{\lambda}(x)$
is a polynomial of degree $|\lambda| = \sum_i \lambda_i$,
and called a Schur function~\cite{macdonald}. 
In infinitely many variables (not our situation) the
Schur function is irreducible~\cite{farahat58}, but
in finitely many variables, the Schur function is sometimes irreducible
and sometimes not (but there are irreducible Schur functions of arbitrarily 
high degree) \cite[Exercise 7.30]{stanleyEnumComb2}.

We can thus derive the following Theorem and Corollary.

\begin{theorem} By Theorem~\ref{Thm_StructuredMatrixImpossibility}, no enumerable set of black-box operations of bounded degree can compute all
Schur functions accurately when the $x_i$ are complex.
\end{theorem}

\begin{corollary} No enumerable set of black-box operations of 
bounded degree or of bounded number of arguments exists 
that will accurately evaluate all minors of complex generalized
Vandermonde matrices in the generic case. \end{corollary}

If we restrict the domain ${\cal D}$ to be nonnegative real numbers,
then the situation changes: The non-negativity of the coefficients of
the Schur functions shows that they are positive in $\cal D$, and
indeed the generalized Vandermonde matrix is totally positive \cite{karlin}.

Combined with the homogeneity of the Schur function, 
Theorem~\ref{thm_positive_homo} implies that the Schur function,
and so determinants (and minors) of totally positive generalized 
Vandermonde matrices can be evaluated accurately in classical arithmetic (and the algorithms mentioned in Section \ref{sec_Plamen} are more efficient than the algorithm used in proving Theorem \ref{thm_positive_homo}.

Now consider a polynomial Vandermonde matrix $V_P$ defined by a
family $\{P_k(x)\}_{k \in \mathbb{N}}$ of polynomials such that 
deg$(P_k) = k$, and $V_P(i,j) = P_{j-1}(x_i)$. Note that these are included in the class of generalized Vandermonde matrices, and that the difference lies in the fact that for polynomial Vandermonde, the sequence of degrees is increasing and without gaps.

Note that any $V_P$ can be written as $V_P = V C$, with $V$ 
being a regular Vandermonde matrix, and $C$ being an upper 
triangular matrix of coefficients of the polynomials $P_k$, i.e., 
\[ 
P_{j-1}(x) = \sum_{i=1}^{j} C(i,j) x^{i-1}~, ~~\forall 1 \leq j \leq n~.
\]
Denote by $c_{i-1} := \tilde{D}(i,i)$, for all $1 \leq i \leq n$ 
the highest-order coefficients of the polynomials $P_0(x), \ldots, P_{n-1}(x)$.

The following two results are proved informally in \cite[Section 5]{focm}.

\begin{theorem} The set of principal minors of polynomial Vandermonde matrices includes polynomials which have irreducible factors of arbitrarily large degree. \end{theorem}

\begin{corollary} By Theorem \ref{Thm_StructuredMatrixImpossibility}, the set of polynomial Vandermonde matrices contains matrices whose inverses cannot be evaluated accurately even with the addition of any enumerable set of bounded-degree black boxes. \end{corollary}

We can also say something about the $LDU$ factorizations of polynomial Vandermonde matrices.
With the matrix $C$ being the upper triangular matrix of coefficients of the polynomials $P_k$, we can write $C = \tilde{D} \tilde{C}$, 
with $\tilde{D}$ being the diagonal matrix of highest-order coefficients, 
i.e., $\tilde{D}(i,i) = C(i,i)$ for all $1 \leq i \leq n$. 
We will assume that the matrices $C$ and $\tilde{D}$ 
are given to us exactly. 

If we let $V_P = L_P D_P U_P$ and $V = LDU$, it follows that 
\begin{eqnarray*}
L_P & = & L~;\\
D_P & = & D \tilde{D}~; \\
U_P & = & \tilde{D}^{-1} UC~.
\end{eqnarray*}

Since we cannot compute $L$ accurately in the general Vandermonde case, 
it follows that we cannot compute $L_P$ accurately in the polynomial 
Vandermonde case.
Likewise, neither the SVD nor the symmetric eigenvalue decomposition (EVD)
are computable accurately, but if the polynomials are certain orthogonal
polynomials, then the accurate SVD is possible \cite{polyvandsvd},
and an accurate symmetric EVD may also be possible \cite{dopicomoleramoro03}.

%

\subsubsection{Positive results: using extended arithmetic} \label{pos_res}

Table~\ref{table1} gathers together structured matrix classes for which it has been established whether and which accurate linear algebra algorithms exist. For some matrix classes, it was deduced that accurate class-algorithms do not exist, from the fact necessary condition (having an accurately evaluable determinant) was violated. 

In this section, we explain how we can use the sufficiency condition for complex matrices developed in Section \ref{suf_c_bb}. 


Consider complex polynomial Cauchy matrices, defined (in their simplest form) as follows. Let $p$ and $q$ be complex polynomials of one variable.   
Let now, using MATLAB notation, 
\begin{eqnarray*}
x_i & \eqbd & p(\widehat{x_i})~, ~~\forall 1 \leq i \leq m \\
y_j & \eqbd & q(\widehat{y_j})~, ~~\forall 1 \leq j \leq m~.
\end{eqnarray*}

\begin{definition} \label{cauchy_m}
We call the matrix $C = (C_{ij})$ with $C_{ij} = \frac{1}{x_i + y_j}$ where $x_i$ and $y_j$ are as above a polynomial Cauchy matrix.
\end{definition}

\begin{definition}
Let 
\begin{eqnarray*} 
Q^{-} (\widehat{x_i}, \widehat{y_j}) & = & p(\widehat{x_i}) -  q(\widehat{y_j})~,\\
Q^{+} (\widehat{x_i}, \widehat{y_j}) & = & p(\widehat{x_i}) +  q(\widehat{y_j})~,
\end{eqnarray*}
be complex polynomials over $\C^{2}$.
\end{definition}

Recall that the determinant of the Cauchy matrix $C$ is
\begin{eqnarray} \label{det_C}
\det C = \frac{\prod_{i,j} (x_i - x_j) (y_i - y_j)}{\prod_{i,j} (x_i+y_j)}~.
\end{eqnarray}

Although our models of arithmetic do not incorporate division, computers do perform division by a non-zero number as an accurate operation. Therefore, given accurate division and black-box algorithms for computing the polynomials $Q^{-}$ and $Q^{+}$, one immediately has a simple and accurate algorithm to evaluate \emph{any} minor for the matrix $C$, therefore any linear algebra operations can be easily performed on $C$ (this algorithm is guaranteed by Theorem \ref{q-suff-c1}). 
 
In fact, we can obtain a much more general result. 

\begin{theorem}
Let $\Phi$ be a formula satisfying NIC and depending on variables $x_1, \ldots, x_n$. Let $p$ be a polynomial (resp., let $\{p_i\}_{1}^{n}$ be a set of polynomials), and let $x_i = p(A(i, 1:m))$ (resp., $p_i(A(i, 1:m))$) for some matrix of parameters $A$.

We can accurately evaluate $\Phi$ on the new set of inputs depending on the parameters of $A$, provided that we build three (resp., $m^2+2m$) black boxes, computing 
\[
\left \{ \begin{array}{l} p \\
Q^{+}(y_1, \ldots, y_n, z_1, \ldots, z_n) = p(y_1, \ldots, y_n) + p(z_1, \ldots, p_n) \\
Q^{-}(y_1, \ldots, y_n, z_1, \ldots, z_n) = p(y_1, \ldots, y_n) - p(z_1, \ldots, p_n) \end{array} \right . ~,
\]
respectively, for all $1 \leq i \leq j \leq m$,
\[
\left \{ \begin{array}{l} p_i \\
Q^{+}_{ij}(y_1, \ldots, y_n, z_1, \ldots, z_n) = p_i(y_1, \ldots, y_n) + p_j(z_1, \ldots, p_n) \\
Q^{-}_{ij}(y_1, \ldots, y_n, z_1, \ldots, z_n) = p_i(y_1, \ldots, y_n) - p_j(z_1, \ldots, p_n) \end{array} \right . ~.
\]
\end{theorem}

\vspace{.2cm}

Another class of matrices which admit accurate linear algebra
algorithms in extended arithmetic are the Green's matrices, which arise
from discrete representations of Sturm-Liouville equations. These matrices
are inverses of irreducible tridiagonal matrices.

Generic Green's matrices have a simple four-vector representation (see,
for example, \cite{ikebe79a}, \cite{nabben99a}), as
\[
F_{i,j} = \left \{ \begin{array}{ll} a_{i} b_j, & ~~~\mbox{if}~ i \geq j\\
                                     c_i d_j, &~~~\mbox{if}~ i <j~
\end{array} \right .
\]
for $\vec{a} = (a_1, \ldots a_n), ~\vec{b} = (b_1, \ldots, b_n),
~\vec{c}=(c_1, \ldots, c_n), ~\vec{d}=(d_1, \ldots, d_n)$, and $1\leq i, j
\leq n$.

The case when $\vec{a}=\vec{c}$ and $\vec{b} = \vec{d}$, i.e., the
symmetric case, has been particularly well-studied (see
\cite{gantmacher_krein}, \cite{karlin}), and we describe it it a bit
more detail.

We use the notation $X \left ( \begin{array}{cccc} i_1 & i_2 & \ldots &
i_p \\ j_1 & j_2 & \ldots & j_p \end{array} \right) $ for the minor of
matrix $X$ corresponding to rows $i_1, \ldots, i_p$ and columns $j_1,
\ldots, j_p$, and $\left | \begin{array}{cc} x & y \\ z & t \end{array}
\right |$ for the determinant $(xt-yz)$.

All minors of symmetric Green's matrices have a simple representation
(following \cite{karlin}) as
\[
G \left ( \begin{array}{cccc} i_1 & i_2 & \ldots & i_p \\ j_1 & j_2 &
\ldots & j_p \end{array} \right ) = a_{k_1} \left | \begin{array}{cc}
a_{k_2} & a_{l_1} \\ b_{k_2} & b_{l_1} \end{array} \right | \left |
\begin{array}{cc} a_{k_3} & a_{l_2} \\ b_{k_3} & b_{l_2} \end{array}
\right | \cdots \left | \begin{array}{cc} a_{k_p} & a_{l_{p-1}} \\ b_{k_p}
& b_{l_{p-1}} \end{array} \right | b_{l_p}~,
\]
where $k_m = \min (i_m, j_m)$ and $l_m = \max (i_m, j_m)$.

Similarly, all minors of generic Green's matrices can be shown (through a
simple inductive argument) to be either $0$ or products of linear and
quadratic factors. Here, by ``linear factor'' we mean a factor of the type
$a_i, ~b_j, c_k$, or $d_l$, and by ``quadratic factor'' we mean a factor
of the type $xt - yz$, with $x, ~y, ~t,$ z being entries of $\vec{a},
~\vec{b}, ~\vec{c}, ~\vec{d}$.

We can then conclude that, given a black box computing $p(x,y,z,t) \eqbd
xt-yz$ accurately, by Theorem \ref{q-suff-c1} one can compute all minors
of generic Green's matrices. Therefore, as was observed in
\cite{demmelkoev99}, one can evaluate all the minors of generic Green's
matrices, and consequently perform linear algebra accurately.

Green's matrices belong to the class of \emph{Hierarchically
semi-separable} or \emph{HSS} matrices. There are many definitions of the
latter, one of them being that HSS matrices of order $k \in \mathbb{N}$ 
are matrices for which any off-diagonal submatrix has rank no bigger than
$k$. Other examples are tridiagonal matrices, banded matrices, inverses of
banded matrices, etc. The HSS matrices are extremely useful as
preconditioners, and arise in many applications. Since
determinants of tridiagonal matrices with independent indeterminates as
entries are irreducible, and tridiagonals are special cases of HSS matrices,
some (and perhaps all) HSS matrices do have irreducible determinants.

Still, we believe that further investigation of the large class of HSS matrices
may yield other examples of subclasses for which simple black-box
operations could be constructed in order to accurately compute 
minors, and therefore, be able to perform linear algebra accurately. 

\section{Other Models of Arithmetic}
\label{sec_OtherModels}

Though the arithmetic models in this paper 
use real (or complex) numbers and rounding errors,
our goal is to draw conclusions about practical finite 
precision computation, i.e., with numbers represented
as finite bit strings (e.g., floating point numbers).
In such a bit model, all rational functions of
the arguments can be computed accurately, even exactly,
because the arguments are rational; the only question 
is cost.  In this section we draw conclusions 
about cost from our analysis.

We would like to quantify our intuition that, for example,
it is much cheaper to accurately compute the determinant
of an $n$-by-$n$ Vandermonde matrix with the familiar formula 
than with Gaussian elimination with sufficiently high precision
arithmetic. We do not mean the difference between
$O(n^2)$ and $O(n^3)$ arithmetic operations, but the
difference in cost between low precision and high precision
arithmetic. To quantify this cost, we need to pick a 
number representation.

We will assume that ``failure'' is not allowed, i.e.,
neither overflow nor underflow is permitted, so that
intermediate (and final) results can grow or
shrink in magnitude as needed to complete the 
computation.

We claim that the natural representation to use is the pair
of integers $(e,m)$ to represent $m \cdot 2^e$,
i.e., binary floating point.
Pros and cons of various number models are discussed in \cite{focm},
but we restrict ourselves here to explaining why we choose
floating as opposed to fixed point, which is also widely used
for analysis (in fixed point, $m \cdot 2^e$ would be represented
using up to $e$ explicit zeros before or after the bits representing
$m$).

One can of course represent the same set of (binary) rational
numbers in both fixed and floating point, but floating
point is much more compressed: It takes about
$\log_2 |e| + \log_2 |m|$ bits to represent $(e,m)$,
but about $|e| + \log_2 |m|$ bits to represent
$m \cdot 2^e$ in fixed point, which is exponentially 
larger. 

First, as a result of this possibly exponentially greater
use of space by fixed point, it is possible for a sequence
of $n$ fixed-point arithmetic operations to take time exponential 
in $n$ (repeated squaring doubles the length of
result at each step, even if only a fixed number of 
most significant bits are kept). 
In contrast, $n$ floating point arithmetic operations 
with fixed relative error take time that grows at worst 
like $O(n^2)$ (attained by repeated squaring again, 
which adds one bit to $e$ at each squaring).
In particular, any of the expressions in earlier sections
of this paper can be evaluated in polynomial time in
the size of the expression, and the size of their floating point 
arguments.

Second, this exponentially greater use of space in fixed-point
means that algorithms can appear ``artificially'' cheaper,
because they are only polynomial in the input size $|e| + \log_2 |m|$,
whereas they would not be polynomial as a function of the 
input size measured as $\log_2 |e| + \log_2 |m|$. 
(This is analogous to asking whether an algorithm with integer
inputs runs in polynomial time or not, depending on whether 
the inputs are represented in unary or binary.)
For example, it is possible to accurately compute the determinant of a 
general matrix with fixed point entries in polynomial time 
in the size of the input \cite{clarkson}, but we know of no 
such polynomial time algorithm with floating point entries.
Running a conventional determinant algorithm (e.g., Gaussian 
elimination with pivoting) in high enough precision would 
require roughly $\log_2  \kappa (A) = \log_2 ( \|A\| \cdot \|A^{-1}\| )$
bits of precision, which can grow like $|e|$ rather than $\log_2 |e|$
(e.g., consider
\[
A = \bmat{cc} y-x & y \\ y & y+x \emat
\]
for $y \gg x$, where $\det (A) = -x^2$).

Indeed, the obvious ``witness'' to identify a singular matrix,
a null-vector, can have exponentially more nonzero bits than
the matrix, as the following example shows. Consider the 
$(2n{+}1)$-by-$(2n{+}1)$
tridiagonal matrix $T$ with 1s on the subdiagonal, $-1$s on the
superdiagonal, and
${\rm diag}(T) = [x_1,x_2,...,x_{n-1},x_n,0,-x_n,-x_{n-1},...,-x_2,-x_1]$.
It is easy to confirm that $T$ is singular, with right null
vector $v = [1,p_1,p_2,...,p_{2n}]$ where
$p_i = {\rm det}(T(1:i,1:i))$ is a leading principal minor.
If we let $x_i = 2^{e_i}$ with $e_1=0$, $e_2=1$, and
$e_i \geq e_{i-1} + e_{i-2}$, then one can confirm 
for $i\leq n$ that $p_i$ is an integer with $f_i$ nonzero bits,
where $f_1=1$, $f_2=2$, and $f_i = f_{i-1} + f_{i-2}$ 
is the Fibonacci sequence. Since $f_i$ grows exponentially,
the null vector $v$ has exponentially many bits as a function
of $n$, whereas the size of $T$ is at most
$O(n \log e_n)$, which can be as small as $O(n^2)$. 

Another way to see the difference between fixed and floating
point is to consider the simple expression $\prod_{i=1}^n (1 + x_i)$.
If the $x_i$ are supplied in fixed point, the entire
expression can be computed exactly in polynomial time.
However in floating point, though the leading bits and trailing
bits are easy, computing some of the bits is
as hard as computing the permanent, a problem
widely believed to have exponential complexity in $n$
\cite{Valiant79}.

Here is the reduction to the permanent.\footnote{We acknowledge 
Benjamin Diament for having discovered the result relating 
floating point complexity to the permanent.} 
Let $A$ be an $n$-by-$n$ matrix whose entries
are $0$s and $1$s. The permanent is the same as the determinant,
except that all terms in the Laplace expansion are added,
instead of some being added and some subtracted. 
Let $r_i$ and $c_j$ be independent indeterminates, and
consider the multivariate polynomial
\begin{equation}
\label{eqn_Permanent}
p(r_1,...,r_n,c_1,...,c_n) = \prod_{A_{ij} \neq 0} (1 + r_i c_j) .
\end{equation}
Then the coefficient $k$ of $\prod_{i=1}^n r_i c_i$ in the
expansion of $p$ can be seen to be the permanent. 
Next we replace $r_i$ and $c_j$ by widely enough 
spaced powers of $2$, so that every coefficient of every term in the
expansion of $p$ appears in non-overlapping bits of $p$ evaluated
at these powers of $2$.
Since no coefficient can exceed $2^{n^2}$, and
since the sequence of exponents $(f_n,...,f_1,e_n,...,e_1)$
in any term $\prod_{i=1}^n r_i^{e_i} c_i^{f_i}$ of $p$
can be thought of as the unique expansion of a number in
base $n+1$,
one can see that choosing $r_i = 2^{n^2(n+1)^{i-1}}$
and $c_j = 2^{n^2(n+1)^{n+j-1}}$ suffices.
The biggest possible product $r_ic_j$ is
$r_nc_n = 2^{n^2((n+1)^{n-1} + (n+1)^{2n-1})} \leq 2^{2n^2(n+1)^{2n}}$,
where the exponent takes at most
$\log_2 (2n^2(n+1)^{2n}) = O(n \log n)$ bits to represent, so
all the arguments $r_ic_j$ in
the product in (\ref{eqn_Permanent}) take $O(n^3 \log n)$ bits
to represent.

Now we consider ``black box arithmetic'',
whose purpose is to model the use of subroutine
libraries with selected high accuracy operations.
We claim that any multivariate polynomial (``black box'') 
with $t$ terms of maximum degree $d$, can be evaluated
accurately in polynomial time as a function of $d$, $t$ 
and the size of the input floating point numbers. The algorithm is 
simply to evaluate each term exactly, and then sum them in decreasing
order of exponents, using a register of about $\log_2 t$ bits
more than needed to store the longer term exactly \cite{demmelkoev04,demmelhida}.
In particular, any enumerable collection of black-boxes
that are all bounded in degree $d$ and number of terms $t$
can all be thought of as running in time polynomial in
the size of their floating point arguments, just like the
basic operations of addition, subtraction and multiplication.
If the number of terms $t$ is proportional to the number
of inputs (e.g., dot products of vectors of length $t$), 
then the cost is still polynomial in the input size.

In summary, in a natural floating point model of arithmetic,
the algorithms we have discussed run in polynomial time in
the size of the inputs, whereas simply running a conventional
algorithm in sufficiently high precision arithmetic to get the
answer accurately can take exponentially longer. We know of no
guaranteed polynomial-time alternatives to our algorithms.

\section{Structured Condition Numbers}
\label{sec_Conditioning}

In this section we begin by recalling some attractive
properties of structured condition numbers for 
problems that we can solve accurately,
and discuss possible generalizations.
If our problem is evaluating the function $p(x_1,...,x_n)$,
then the structured condition number $\kappa_{struct}$ is simply
the derivative of the relative change in $p$ with respect
to relative changes in its arguments:
\begin{equation}
\label{eqn_StructCond}
\kappa_{struct} = 
\frac{\| (x_1 \frac{\partial p}{\partial x_1} , ... , 
          x_n \frac{\partial p}{\partial x_n}) \|}
{|p|}
\end{equation}
where any vector norm may be used in the numerator.

The simplest case, as before, is 
for problems described by
Theorem~5.12 and Corollary~5.15, which say that
in the complex case, a necessary and sufficient
condition for accurate evaluation of complex $p(x)$
using only traditional arithmetic ($\pm$ and $\times$) 
is that $V(p)$ be allowable, in which case $p(x)$ factors
completely into factors of the forms $x_i^\alpha$,
and $(x_i \pm x_j)^{\beta}$, where
$\alpha$ and $\beta$ are fixed integers.
This covers many of the linear algebra examples
in Section~\ref{sec_Plamen}. Given such a simple expression it
is easy to evaluate the structured condition number:
Each factor $x_i^{\alpha}$ adds $\alpha$ to 
$\frac{x_i \frac{\partial p}{\partial x_i}}{p}$, and each factor
$(x_i \pm x_j )^{\beta}$ adds 
$| \beta x_i / (x_i \pm x_j) | \leq | \beta | / \relgap(x_i,\mp x_j)$.

Slightly more generally, for expressions satisfying NIC,
e.g., including real expressions that only add like-signed values,
analogous conclusions can be drawn. This is because factors that
only add like-signed values can only make bounded contributions
to the condition number.

Given a structured condition number for a decomposition
like LDU with complete pivoting (an RRD), this essentially
becomes a structured condition number for the SVD
\cite[Thm 2.1]{DGESVD}.

Now we consider the set of {\em ill-posed problems}, i.e., the
ones whose structured condition numbers are infinite. 
Examining (\ref{eqn_StructCond}), we see that $p=0$ is
a necessary condition, i.e., the ill-posed problems
are a subset of $V(p)$. (If $p(x)$ were rational, we
would include the poles as well.)
For every term $| \beta | / \relgap(x_i,\mp x_j)$ 
in the structured condition number, the corresponding ill-posed set is
defined by $x_i = \mp x_j$.
All of $V(p)$ is not necessarily ill-posed, since for
example small relative changes in $x$ only cause small
relative changes in $p(x) = x^{\alpha}$.

It is natural to ask if there is a relationship between the 
{\em distance to the nearest ill-posed problem}, i.e., 
the smallest relative change to the $x_i$ that make the 
problem ill-posed, and its structured condition number
\cite{demmel87}.
It is easy to see that for any term
$| \beta | / \relgap(x_i,\mp x_j)$ in the structured condition number,
the smallest relative changes to $x_i$ and $\mp x_j$ that make it
infinite are close to $\relgap (x_i, \mp x_j)$ when it is
small. In other words, the structured condition number is close
to the reciprocal of the distance to the nearest ill-posed problem,
measured by the smallest relative change to the arguments $x_i$.
This helps explain geometrically why the structured condition number
can be so much smaller than the unstructured one: it takes, for example,
a much larger perturbation to 
make $x_i = i-\frac{1}{2}$ and $x_j = j-\frac{1}{2}$ equal
than the smallest singular value of the Hilbert matrix $H_{ij} = 1/(x_i + x_j)$.

This reciprocal-condition-number property, that the reciprocal of the
condition number is approximately the distance to the nearest ill-posed
problem, is common in numerical analysis \cite{demmel87,rump99a,rump03a}.
The following simple asymptotic argument shows why:

If the structured condition number (\ref{eqn_StructCond}) is very large,
then some component 
$|x_i \frac{\partial p}{\partial x_i} / p | \gg 1$,
i.e., 
$|p / \frac{\partial p}{\partial x_i} | \ll |x_i|$,
or in other words one step of Newton's method
$x_i^{new} = x_i - p / \frac{\partial p}{\partial x_i}$
to find a root of $p=0$ will take a very small step. Therefore 
it is plausible that this step 
$p / \frac{\partial p}{\partial x_i}$
is very close to the smallest 
(absolute) distance to the variety in the $x_i$ direction
(or an integer multiple of 
$p / \frac{\partial p}{\partial x_i}$
is, the multiplicity of the root)
and dividing by $|x_i|$ yields the relative distance.

Now let us go beyond expressions evaluable accurately just using NIC.
Consider the case of a real positive polynomial or empty variety, 
as discussed in Section~\ref{sec_PositivePolys}. The analysis in 
Theorem~\ref{thm_positive_compact} (resp., Theorem~\ref{thm_positive_homo}) 
shows that the relative condition number 
will grow like $1/p_{min}$ (resp., $1/p_{min,homo}$), the reciprocal
of the smallest value $p(x)$ can take on the appropriate domain. 
So the relative condition number can be arbitrarily large, 
but 
in the absence of a variety intersecting the domain
it remains bounded.

Based on these examples and analysis, we conjecture that for 
traditional arithmetic, the following two statements hold.
(1) The reciprocal of the structured condition number is an approximation
of the relative distance from $x$ to the nearest ill-posed problem, 
perhaps asymptotically.
(2) This relative distance is approximately given by $\relgap (x_i, \mp x_j)$ 
for some $i$ and $j$.

This reciprocal-condition-number
property is quite robust as the arguments above suggest,
and does not necessarily depend on accurate evaluability.
For example, if $p(x) = (x_1 + x_2 + x_3)^{\alpha}$ then
its structured condition number is $\alpha \| x \| / |x_1 + x_2 + x_3|$,
and $|x_1 + x_2 + x_3| / \|x\|_1$ is indeed the relative distance.
However, the reciprocal-condition-number property is not universal but depends
on the structure we impose \cite{rump98,rump99b,rump03b}. 
Just as this reciprocal condition number property is equivalent to 
the statement that computing the condition number is as sensitive a 
problem as solving the original problem, we conjecture that the structured 
condition number $\kappa_{struct}$ can only be computed accurately if the 
original problem $p$ can be, at least in the interesting case
when $\kappa_{struct}$ is large.
This seems reasonable since $p(x)$ ends up in the denominator of
$\kappa_{struct}$, so we need to evaluate $p$ accurately near its zeros 
(or poles). But the numerators $\partial p / \partial x_i$ could be anything, and perhaps
even have zeros on unallowable varieties, so to be more precise
we conjecture that $p$ can be evaluated accurately in some open
neighborhood of its zeros (or poles) if and only if
$\kappa_{struct}$ can be.


\section{Conclusions}  \label{sec_Conclusions}

In this paper, we have made the case for accurate evaluation of polynomial expressions and accurate linear algebra; we have shown that such evaluation is desirable (Section \ref{sec_intro}), significant (Section \ref{sec_OtherModels}) and often realizable efficiently (Section \ref{sec_Plamen}). We have listed, in Section \ref{sec_Plamen}, many types of structured matrices that have been analyzed from an accuracy perspective in the numerical linear algebra literature, while in Section \ref{sec_OlgaIoana} we identified the common algebraic structure that made them analyzable in the first place. 

There are limits to how much we can hope to extend the class of structured matrices for which linear algebra 
can be performed accurately; the ``negative examples'' of Section~\ref{sec_conseq} show that, for some classes of matrices, accuracy cannot be achieved in finite precision, and both Sections \ref{sec_Plamen} and \ref{sec_OlgaIoana} mention problems that are impossible to solve in ``traditional'' arithmetic. The former should be seen as ``hard'' barriers, but the latter should be seen as a challenge, both from theoretical and computational perspectives. The theory  should aim to provide answers to the question of how to extend one's arithmetic by adding ``black-box'' operations, in order to make these structured problems solvable (as we do for the examples of Section \ref{sec_design}); the computation should design software implementing such ``black boxes''. 

In summary, accurate evaluation is an important area of scientific computing, which has been advanced by the recent results presented here. Plenty of work remains  in adding to both the theoretical framework (which apparently requires familiarity with ``pure'' mathematical fields such as algebraic geometry, topology, and analysis) and to the practical one (software implementation). 


\bibliographystyle{plain}
\bibliography{biblio}

\end{document}